\documentclass[letter,11pt]{amsart}
\usepackage{amsmath,amsthm,amssymb,amsfonts,enumerate,color,hyperref,esint, graphicx, pgf,tikz}
\usepackage{hyperref}

\usepackage{amssymb,amsfonts}
\usepackage[all,arc]{xy}
\usepackage{enumerate}
\usepackage{comment}
\usepackage{mathrsfs}
\usepackage[makeroom]{cancel}
\usepackage{orcidlink}
\usepackage{enumitem}
\usepackage{tikz-cd}

\usepackage{mathtools}
\usepackage{esint}

\newtheorem{thm}{Theorem}[section]
\newtheorem{cor}[thm]{Corollary}
\newtheorem{prop}[thm]{Proposition}
\newtheorem{lem}[thm]{Lemma}

\theoremstyle{definition}
\newtheorem{defn}[thm]{Definition}

\newtheorem{notns}[thm]{Notations}

\theoremstyle{remark}
\newtheorem{rem}[thm]{Remark}

\newcommand{\diag}{\text{\rm diag}}

\newcommand{\eps}{{\varepsilon}}

\newcommand{\RR}{{\mathbb R}}

\def\phi{\varphi}
\def\eps{\epsilon }

\def\D{\partial }
\newcommand\adots{\mathinner{\mkern2mu\raise1pt\hbox{.}
\mkern3mu\raise4pt\hbox{.}\mkern1mu\raise7pt\hbox{.}}}

\renewenvironment{align}{
    \begin{equation}
    \begin{aligned}
	}
	{
    \end{aligned}
    \end{equation}
    \ignorespacesafterend
}
\numberwithin{equation}{section}
\mathtoolsset{showonlyrefs=true}
\title{Free Boundary Regularity for Non-Convex Fully Nonlinear Alt-Phillips Problems}

\author[A.~Zahl]{Alvis Zahl \orcidlink{0000-0001-6390-5913}}
    \address{Alvis Zahl. Department of Mathematics\\
    Rutgers University\\
110 Frelinghuysen Rd., Piscataway, NJ 08854, USA}
    \email{azahl@asu.edu}

\keywords{Free Boundary Problems, Fully Nonlinear Equations, Hodograph Transform}

\subjclass[2020]{Primary: 35R35, 35A22. Secondary: 35J60, 35J70.}
\thanks{Alvis Zahl, Department of Mathematics, Rutgers University, 110 Frelinghuysen Rd.,
Piscataway, NJ 08854, USA. \emph{Email:} azahl@asu.edu ORCiD:0000-0001-6390-5913}

\calclayout

\begin{document}

\begin{abstract}
This paper studies the fully nonlinear Alt-Phillips free boundary problem
\begin{equation} \nonumber
F(D^2u) = u^\gamma \chi_{\{u>0\}},
\end{equation}
for $\gamma \in (-1/3,1)$ without any convexity assumption on $F$. We establish that flat free boundaries are smooth. This result is new even for the fully nonlinear obstacle problem when $\gamma = 0$.  

Our approach is based on a partial hodograph transform, which converts the free boundary into a fixed boundary and produces a fully nonlinear degenerate equation. For this equation we prove a Harnack inequality, yielding an improvement of flatness, and a Schauder estimate for the associated linearized equation. Both estimates appear to be new and are of independent interest.
\end{abstract}

\maketitle

\section{Introduction}
Let $\Omega$ be a domain in $\RR^n$. We study the fully nonlinear Alt-Phillips problem:
\begin{align} \label{problem}
\begin{cases}
F(D^2 u) = u^\gamma \chi_{\{u>0\}}, \\
u \geq 0,
\end{cases}
 & \text{in } \Omega.
\end{align}

Here $\gamma \in (-1,1)$, $F$ is a real valued function defined on the space of real $n\times n$ symmetric matrices, and $F$ is uniformly elliptic.

The Alt-Phillips problem, originally introduced as the Euler-Lagrange equation of a variational functional by Alt and Phillips \cite{Alt1986}, interpolates between the two most studied free boundary problems: the Bernoulli problem ($\gamma = -1$) and the obstacle problem  ($\gamma = 0$). For $\gamma \geq 1$, \eqref{problem} yields a strong maximum principle and has no free boundary, while for $\gamma \leq -1$, the problem has very different monotonicity properties and free boundary behavior, which is not studied here.  This problem also arises in the study of gas distribution in reaction with a porous catalyst pellet where the solution $u$ models the density of the gas \cite{aris1975mathematical}. 

In recent years, attention has turned to the fully nonlinear analogues. The fully nonlinear obstacle problem ($\gamma = 0$) has been studied extensively in \cite{MR3648978, Figalli2014, MR3957397, MR3542613, Lee, LeePark}. For general $\gamma \in (-1,0) \cup (0,1)$, Wu and Yu \cite{WuYu} established the optimal regularity of the solution $u$, and the $C^1$ regularity of the regular part of the free boundary. 

All these works above require some convexity assumptions on the operator $F$, either directly or as in \cite{LeePark} through structural conditions on $F$ as in \cite{CaffarelliYuan} that play an equivalent role. Convexity is often used in a structural way to enable blow up argument in spirit of \cite{Caffarelli1970}. 

There is, however, no evidence that convexity is the natural assumption in studying the regularity of the free boundary. This paper makes no convexity assumption and establishes that flat free boundaries are smooth, showing that the free boundary regularity is intrinsic to uniform ellipticity rather than an artifact of convexity. The result is new even in the case of the fully nonlinear obstacle problem when $\gamma = 0$.

We assume $F$ satisfies the following.

\begin{enumerate}
\item[(H1)] \label{H1} $F \in C^\infty(\mathcal{S})$, where $\mathcal{S}$ is the space of real $n \times n$ symmetric matrices, $F(0) = 0$.
\item[(H2)] $F$ is uniformly elliptic: there exist positive constants $\lambda$, $\Lambda$ such that for any $M \in \mathcal{S}$ and $\xi \in \RR^n$, 
\begin{equation} \nonumber
\lambda|\xi|^2 \leq \xi^T (\D F(M)/\D r_{ij}) \xi \leq \Lambda |\xi|^2
\end{equation}
where $r_{ij}$ denotes the $ij$-th component of matrix $M$.
\item[(H3)] \label{H3} When $\gamma<0$, for any symmetric matrix $M$ with $\|M\| \leq 1$, 
\begin{equation}\nonumber
\lim_{t\rightarrow \infty}\frac{1}{t}F(tM) = F_\infty(M),
\end{equation}
and there exists a constant $C$ and $t_0$ such that for all $t > t_0$, 
\begin{equation}
|tD^2F(tM)| \leq C.
\end{equation}

%\item[(H4)] $F(0) = 0$.
\end{enumerate}

We say that $x_0\in \D \{u > 0\}$ is $\eps$-flat if after possible translation and rotation, there exists fixed  $\eps \in (0,\eps_0)$ such that for $x \in B_1(x_0)$,
\begin{equation}\label{flat}
h(x\cdot e_n-\eps) \leq u(x) \leq h(x\cdot e_n+\eps). 
\end{equation}
where $e_n$ is the unit vector in the normal direction. 

The main result of the paper is the following theorem.

\begin{thm}\label{ThmA}
Let $u$ be a viscosity solution to \eqref{problem}. For $\gamma \in (-1/3,1)$, there exists $\eps > 0$, such that if $\Gamma  = B_{1}(x_0)\cap \D \{u > 0\}$ is $\eps$-flat, then $\Gamma \cap B_{1/2}(x_0)$ is the graph of a $C^\infty$ function $g$, and for every $k \in \mathbb{N}$,
\begin{equation}
\|g\|_{C^k(B'_{1/2}(x'_0))} \le C_k
\end{equation}
with $C_k$ depending only on $n, \gamma, \lambda, \Lambda$ and $k$.
\end{thm}

We note that if we assume, instead of (H1), that $F \in C^{k}$ for some $k \geq 2$, then the same proof shows that the flat free boundary is $C^{k+1,\beta}$ for some $\beta \in (0,1)$.

Inspired by Allen, Kriventsov, and Shahgholian \cite{Dennis2025}, our approach is to study the problem using a modified partial Hodograph transform developed by Kinderlehrer, Nirenberg in \cite{Kinderlehrer}, Kinderlehrer, Nirenberg, and Spruck in \cite{KinderlehrerSpruck} to analyze the behavior of the free boundary. The hodograph transformation converts the free boundary into a fixed boundary and maps the problem to a highly degenerate fully nonlinear elliptic equation. Similar techniques also appear in the work of De Silva, Forcillo, and Savin \cite{MR4308249} for the one-phase Stephan problem.

While higher regularity of the free boundary for the obstacle problem was originally obtained with the Legendre transform \cite{Petrosyan}, it is not suitable for the Alt-Phillips problem. Specifically, adopting the normal derivative $u_n$ as coordinates breaks the essential level set structure and the hodograph transform does not.

We want to point out that the proof relies on two major technical contributions, each of which yields a regularity result of independent interest for degenerate elliptic equations.

First, under the flatness assumption (\ref{flat}), $u$ is mapped to some $v$ which is possibly multivalued near $\{x_n = 0\}$. In the case $\gamma > -\frac{1}{3}$, we introduce an appropriate notion of viscosity solution to be compatible with multivalued functions and first establish a Harnack type inequality which yields an improvement of flatness. The argument requires nontrivial modifications of the argument in \cite{DeSilva2011} due to the nonlinearity and the degeneracy of the equation. By constructing a particular barrier, we prove the regularity all the way to the boundary $\{x_n = 0\}$ where the equation degenerates. Beyond its role in the present paper, this provides, to our knowledge, the first Harnack inequality for fully nonlinear elliptic operators with this type of boundary degeneracy.

Secondly, by differentiating the equation in the tangential directions, we get a degenerate linear equation, where the coefficients depend on the solution itself. We first establish a Liouville theorem to the frozen coefficient equation and then prove a Schauder estimate in the style of \cite{Simon} for this degenerate equation. By iterating this argument, it follows that the flat free boundary is $C^\infty$. Beyond the application to free boundary regularity, there has been recent interests and developments in the regularities for such degenerate elliptic equations by Dong and Kim \cite{DongKim, DongKimRevisit}, Dong and Phan \cite{DongPhan}, and Sire, Terracini and Vita \cite{Terracini1,Terracini2}, among many other works. The Schauder type estimate developed in this paper extends current results and applies in the case of non-divergence form degenerate equations with variable coefficients. 

We point out that the critical value of the exponent $\gamma = -1/3$ appears when the free boundary changes behavior and we briefly explain here. In the hodograph variable, the solution $v$ satisfies a linearized equation and in leading order in $x_n$ behaves like the solution to the ODE of the form
\begin{equation}\nonumber
(x_n^\alpha v_n)_n = 0 \text{ in } {x_n > 0}
\end{equation}
where the two solutions are $C_1x_n^{1-\alpha}$ and $C_2$. When $\gamma > -1/3$, $\alpha > 1$, and $x_n^{1-\alpha}$ is unbounded near $x_n = 0$, and thus this term has to vanish. Equivalently, the half space solutions satisfies a Neumann type boundary condition $x_n^\alpha u_n \rightarrow 0$ as $x_n \rightarrow 0$, and the solution can be extended via even reflection across $\{x_n = 0\}$ and solves the same equation where the coefficients are extended correspondingly. 

In the case $\gamma \leq -1/3$, the approach in the paper also applies, but would require different elliptic regularity theory for the corresponding oblique derivative problems and this is not pursued here.

We conclude the discussion of Theorem \ref{ThmA} with a few general remarks. First, it is a priori not obvious that the free boundaries are smooth in this general setting. The mechanism here is in spirit of Savin's small perturbation theorem \cite{savin} where the solution is close to a 1D solution. 

Theorem \ref{ThmA} suggests that for fully nonlinear free boundary problems, the bulk of the free boundaries are actually better behaved than the interior. This motivates the study of the free boundaries in this non-convex class of fully nonlinear equations. A natural next step is to study the singular sets of the free boundary. In particular, for $\gamma = 0$, the structure of the the singular set remains completely open. 

There are other interesting classes of operators beyond the ones studied here. For example, instead of assuming $F$ is smooth, one may assume that $F$ is $1$-homogeneous, for which the Pucci extremal operators are model cases. Such operators admit no linearization around the one-dimensional solution. Some free boundary regularity is available for the convex extremal case. Thus it is natural to explore the concave extremal case and general operators with such structured non-smooth behavior along one-dimensional solutions.

The rest of the paper is organized as follows. Section \ref{sec:prelimiary} establishes properties of the one-dimensional solution and preliminary results that follow directly from flatness. Section \ref{Sec:Hodograph} defines the Hodograph transform and derives the PDE in the transformed variables. Section \ref{Sec:Harnack} proves a Harnack type inequality. Section \ref{Sec:Liouville} focuses on the analysis of the constant coefficient linearized equation, where we establish a Liouville type theorem and a Schauder theorem. Section \ref{Sec:Improvement} shows the free boundary is $C^{1,\sigma}$ and Section \ref{Sec:Higher} improves this to $C^{2,\sigma}$, and iterates to $C^\infty$.

We finally introduce the following notations to be used throughout this paper.
\begin{notns}\hfill
\begin{enumerate} 
\item We denote partial derivatives by subscript. That is, 
\begin{equation}\nonumber
u_i = \D_i u = \frac{\D}{\D x_{i}} u.
\end{equation}
\item For a function $F(r,p): S \times \RR^n \rightarrow \RR$, where $S$ is the space of real symmetric matrices, we denote $r = (r^{ij}) \in S$, and $p = (p^i) \in \RR^n$. 
\item We denote \begin{equation}\nonumber
\Omega_+ = \{u > 0\},\quad \Gamma = \D \{u > 0\}.
\end{equation}
\item $\RR^{n,+} = \RR^n \cap \{x_n >0\}$, $B_r(x)^+ = B_r(x) \cap \{x_n >0\}$.
\item We use the convention of summing over repeated indexes.
\item For $x\in \RR^n$, denote $x = (x',x_n)$ where $x' \in \RR^{n-1}$. 
\item Let $f,g$ be two real valued functions. We write $f \sim g$ if there exists constants $c$, $C$, such that $cf(x) \leq g(x) \leq Cf(x)$ for all $x$.
\item Let $h(x_n): \RR \rightarrow \RR$. With abuse of notation, we use $D^2h$ as the matrix \begin{equation}
D^2h = \diag(0,\cdots,0,h'').
\end{equation}

\item Let $A = (a_{ij})$, $B = (b_{ij})$, $1\leq i,j\leq n$ be two matrix, we denote 
\begin{equation}\nonumber
A: B = \sum_{i,j=1}^n a_{ij} b_{ij}
\end{equation}
Let $p= (p_1,\cdots, p_n), q= (q_1,\cdots, q_n) \in \RR^n$ be vectors, we denote
\begin{equation}\nonumber
[A,p]:[B,q] = \sum_{i,j=1}^n a_{ij} b_{ij} + p_iq_i.
\end{equation}
\item We use $\sigma, \beta \in (0,1)$ as H\"older exponents which might vary at different places.
\end{enumerate} 
\end{notns}

\section{Preliminaries}\label{sec:prelimiary}

Let $h = h(x_n)$ be the unique positive one-dimensional solution to the ODE
\begin{align} \label{1dpde}
\begin{cases}
F(D^2h) = h^\gamma\\
h(0) = h'(0) = 0.
\end{cases}
\end{align}

We note that \eqref{1dpde} doesn't yield unique solutions since $h(t) = 0$ is a solution, and $h(t-t_0)$ is another whenever $h$ is a solution and $t_0 >0$. We prove that there is a unique solution that is positive for $t >0$. The following Proposition follows similarly from the semilinear case in \cite{Dennis2025}.

\begin{prop} \label{hprop} 
There exists a unique positive solution $h$ to \eqref{1dpde}, and $h$ satisfies the following properties. There exists constant $C, C_1,C_2 >0$ depending on $\gamma, \lambda, \Lambda$, such that
\begin{align}\nonumber
&C_1t^{2/1-\gamma} \leq h(t) \leq C_2t^{2/1-\gamma} &&\text{ for } 0<t<1,\\
& {C_1}\leq t\frac{h'(t)}{h(t)}\leq {C_2}&&\text{ for } 0<t<1,\\
& C_1 \leq t\frac{h''(t)}{h'(t)}\leq C_2 &&\text{ for } 0<t<1,\\
& |h'''| \leq C\frac{h}{t^3} &&\text{ for } 0<t<1,\\
&K^{C_1}t \leq h^{-1}(Kh(t))\leq K^{C_2}t &&\text{ for } K > 0.\\
\end{align}
\end{prop}
\begin{proof}
Since $F$ is uniformly elliptic and $F(0) = 0$, 
\begin{equation}\nonumber
\lambda \|D^2h\| \leq F(D^2h) \leq \Lambda\|D^2h\|.
\end{equation}
Since $h \geq 0$, $\|D^2h\| = h''$. Using $F(D^2h) = h^\gamma$, we rewrite the above as
\begin{equation}\label{h''bound}
\frac{1}{\Lambda}h^\gamma \leq h''\leq \frac{1}{\lambda}h^\gamma.
\end{equation}
We note that a solution $\tilde{h}$ to 
\begin{align} \nonumber
\begin{cases}
&\tilde{h}'' = c\tilde{h}^\gamma\\
&\tilde{h}(0) = \tilde{h}'(0) = 0
\end{cases}
\end{align}
is given explicitly by
\begin{equation}\nonumber
\tilde{h}(t) = \left[c\frac{(1-\gamma)^2}{4\gamma}\right]^{1/1-\gamma}t^{\frac{2}{1-\gamma}}.
\end{equation}
Thus,
\begin{equation} \label{explicithbound}
C_1\tilde{h}(t) \leq h(t) \leq C_2\tilde{h}(t).
\end{equation}

By the method of sub and supersolutions, there exists a positive solution $h$ to \eqref{1dpde} on $0 \le t \le 1$ satisfying $C_1 t^{2/(1-\gamma)} \le h(t) \le C_2 t^{2/(1-\gamma)}$, this is the first inequality in the Proposition.

We now show there is a maximal such positive solution $h_{\max}$. Indeed, consider the solution to
\begin{align}\nonumber
\begin{cases}
y'' = h(y),\\
y(0) = \delta,\quad y'(0) = 0,
\end{cases}
\end{align}
which is bounded below by $h$. As $\delta \to 0$ we obtain $h_{\max}$, and any positive solution is bounded from below by $h_{\max}(t - \delta)$. Then $h = h_{\max}$ is the unique positive solution.

Let $v = h'$, then $v$ satisfies 
\begin{align} \nonumber
\begin{cases}
&F(\text{diag}\{0,\cdots,0,v'\}) = {h}^\gamma\\
&v(0) = 0.
\end{cases}
\end{align}
Using uniform ellipticity as before, we have
\begin{equation}\nonumber
\frac{h^\gamma}{\Lambda} \leq v' \leq \frac{h^\gamma}{\lambda}.
\end{equation}
Integrating the above inequality, and using \eqref{explicithbound}, we have
\begin{equation} \label{h'explicitbound}
C_1t^{\frac{1+\gamma}{1-\gamma}} \leq v \leq C_2t^{\frac{1+\gamma}{1-\gamma}},
\end{equation}
which gives
\begin{equation}\nonumber
{C_1}\leq t\frac{h'(t)}{h(t)}\leq {C_2}.
\end{equation}

For $th''/h'$, we have
\begin{equation}\nonumber
t\frac{h''}{h'} \leq Ct\frac{h^\gamma}{h'} \leq Ct\frac{t^{\frac{2\gamma}{1-\gamma}}}{t^{\frac{1+\gamma}{1-\gamma}}} \leq C_2.
\end{equation}
The lower bound follows similarly. 

We note that $h$ is smooth and we  can differentiate \eqref{1dpde} in $x_n$ direction and get
\begin{equation}\nonumber
\frac{\D F}{\D r^{nn}} (D^2h) h''' = \gamma h^{\gamma-1}h'.
\end{equation}
By uniform ellipticity of $F$,
\begin{equation}\nonumber
\lambda \leq \left|\frac{\D F}{\D r^{nn}} (D^2h)\right| \leq \Lambda 
\end{equation}
Therefore,
\begin{equation}\nonumber
|h'''| \leq C h^{\gamma-1}h' \leq C\frac{h}{t^3}.
\end{equation}
For any $K>0$, we have
\begin{equation}\nonumber
Kh(t) \leq K\tilde{h}(t) \leq CKt^{\frac{2}{1-\gamma}} \leq K^{\frac{2C}{1-\gamma}}t^{\frac{2}{1-\gamma}} \leq h(K^{C_2t}).
\end{equation}
by choosing $C$ big if $K >1$, and $C$ small if $K <1$. The other inequality follows similarly, and we yield
\begin{equation}\nonumber
h(K^{C_1}t)\leq Kh(t) \leq h(K^{C_2}t).
\end{equation}
Apply $h^{-1}$ to all expressions, we obtain the final inequality in the proposition. 
\end{proof}

\begin{prop} \label{hconstant} Let $h$ be the solution to \eqref{1dpde}, then for $\gamma \in (-1,1)$,
\begin{equation}\nonumber
\lim_{x_n \rightarrow 0}\frac{h'' x_n}{h'} = \frac{1+\gamma}{1-\gamma}.
\end{equation}
\end{prop}
\begin{proof}
\textit{Case 1: $\gamma \geq 0$}. Since $F$ is smooth and $F(0) = 0$, we have
\begin{align}
F(M) = F(0) + DF(0):M + O(\|M\|^2).
\end{align}
Plug in $M = D^2h$, since $F(D^2h) = h^\gamma$, we have
\begin{equation}
h^\gamma = C h'' + O(\|D^2h\|^2) .
\end{equation}
Since $h'' \leq Cx_n^{\frac{2\gamma}{1-\gamma}}$, we have as $x_n \to 0$,
\begin{equation}\label{eq:h_second_deriv_asymp}
    h'' = C h^\gamma \big(1 + o(1)\big).
\end{equation}
Multiplying by $h'$ and integrating from $0$ yields
\begin{align}
    h' h'' &= C h^\gamma h' \big(1 + o(1)\big) \nonumber, \\
    \frac{1}{2}(h')^2 &= \frac{C}{\gamma+1} h^{\gamma+1} \big(1 + o(1)\big).
\end{align}
Taking the square root yields
\begin{equation}\label{eq:h_first_deriv_asymp}
    h' = \sqrt{\frac{2C}{\gamma+1}} \, h^{\frac{\gamma+1}{2}} \big(1 + o(1)\big).
\end{equation}
Rearranging the terms, we have
\begin{equation}
    h^{-\frac{\gamma+1}{2}} h' = \sqrt{\frac{2C}{\gamma+1}} \big(1 + o(1)\big),
\end{equation}
and integrating from $x_n = 0$ provides
\begin{equation}
    \frac{2}{1-\gamma} h^{\frac{1-\gamma}{2}} = \sqrt{\frac{2C}{\gamma+1}} \, x_n \big(1 + o(1)\big).
\end{equation}
Rewriting the above equation gives
\begin{equation}\label{eq:xn_asymp}
    x_n = \frac{2}{1-\gamma} \sqrt{\frac{\gamma+1}{2C}} \, h^{\frac{1-\gamma}{2}} \big(1 + o(1)\big).
\end{equation}
Combining all the above, we have
\begin{align}
    \frac{x_n h''}{h'} &= \frac{ \left[ \frac{2}{1-\gamma} \sqrt{\frac{\gamma+1}{2C}} \, h^{\frac{1-\gamma}{2}} \right] \cdot \left[ C h^\gamma \right] }{ \left[ \sqrt{\frac{2C}{\gamma+1}} \, h^{\frac{\gamma+1}{2}} \right] } \big(1 + o(1)\big) \nonumber \\
    &= \left( \frac{2}{1-\gamma} \right) \left( \frac{\gamma+1}{2} \right) \big(1 + o(1)\big) \nonumber \\
    &= \frac{1+\gamma}{1-\gamma} \big(1 + o(1)\big).
\end{align}
Taking the limit as $x_n \to 0$ gives
\begin{equation}
    \lim_{x_n \to 0} \frac{x_n h''}{h'} = \frac{1+\gamma}{1-\gamma}.
\end{equation}

\textit{Case 2: $\gamma \in (-1,0)$}. We note that
\begin{equation}
\frac{1}{h''}F(h'') = \frac{h^\gamma}{h''}.
\end{equation}
By assumption (H3), use $t = h''$, $M = e_n \otimes e_n$, we have 
\begin{equation}
\lim_{x_n \rightarrow 0} \frac{1}{h''}F(h'') = F_\infty(e_n \otimes e_n) = C.
\end{equation}
Rewriting, this gives
\begin{equation}
h^\gamma = Ch'' + o(1).
\end{equation}
By the same calculation as in the previous case yields the result of the Proposition. 
\end{proof}

We next show that $u$ is smooth away from the free boundary by Savin's theorem \cite{savin}. Let $q(x')$ denote the parametrization of the free boundary. 

\begin{thm} \label{Savin}
If $u$ is a viscosity solution to \eqref{problem}, then there exists $\delta = \delta(\eps)$, some $\sigma >0$, such that
\begin{equation}
u \in C^{2,\sigma}(B_{1}\cap\{x_n - q(x')> \delta\}).
\end{equation} 
\end{thm}
\begin{proof}

For any $x_0 \in B_{1}\cap\{x_n - q(x')> \delta\}$ with $\delta$ to be chosen later, choose $r = \frac{1}{2} (x_{0,n} - q(x_0')) > \delta/2$. Let $\kappa = \frac{2}{1-\gamma}$, and define
\begin{align}
u_r(y) &= r^{-\kappa}u(x_0+ry)\\
h_r(y) &= r^{-\kappa}h(x_0+ry).
\end{align}
where $y = {r^{-1}(x-x_0)} \in B_1(0)$. We note that $\|h_r\|$ is uniformly bounded above and below, and 
\begin{align}
D^2u &= r^{\kappa-2}D^2u_r,\\
D^2h &= r^{\kappa-2}D^2h_r.
\end{align}
Define $L(M,z)$ by
\begin{equation}
L(D^2v,v) := r^{2-\kappa}F(r^{\kappa-2}D^2v + r^{\kappa-2}D^2h_r) - (v+h_r)^\gamma. 
\end{equation}
Then we have
\begin{align}
L(0,0) &= r^{2-\kappa}F(r^{\kappa-2}D^2h_r) - (h_r)^\gamma = 0,\\
L(D^2(u_r-h_r),u_r-h_r) &= r^{2-\kappa}F(r^{\kappa-2}D^2(u_r-h_r) + r^{\kappa-2}D^2h_r) - (u_r-h_r+h_r)^\gamma = 0.
\end{align}
It follows from \eqref{flat} that
\begin{equation}
h(x_{0,n}+ry_n - \eps) \leq u(x_{0,n}+ry_n) \leq h(x_{0,n}+ry_n + \eps).
\end{equation}
Thus, using Proposition \ref{hprop} and $r \leq x_{0,n}-q(x')+ry_n \leq 3r$,
\begin{equation}
|u(x_{0,n}-q(x')+ry_n) - h(x_{0,n}-q(x')+ry_n)| 
\leq \varepsilon \sup_{y_n\in[-1,1]} h'(x_{0,n}-q(x')+ry_n) 
\leq C\varepsilon r^{\kappa-1}.
\end{equation}
Dividing by $r^{\kappa}$ yields
\begin{equation}
\|u_r - h_r\|_{L^\infty(B_1)} \leq \frac{C\varepsilon}{r}.
\end{equation}
Choosing $\delta = C_0\varepsilon$ for some $C_0$, such that $r \geq \frac{2}{3}\delta$ gives
\begin{equation}
\|u_r - h_r\|_{L^\infty(B_1)} \leq \eps_0
\end{equation}
where $\eps_0$ is the universal smallness constant in Savin's theorem.
We proceed to check that $\|D^2L\|$ is uniformly bounded in a neighborhood of $0$. We first note that for $\gamma \geq 0$,
\begin{equation}
\|r^{\kappa-2}D^2h_r\| = \|D^2h\| \leq \lambda^{-1}F(D^2h) \leq Ch^\gamma \leq Cr^{\frac{2\gamma}{1-\gamma}} \leq C.
\end{equation}
Since $D^2F$ is continuous, and the set 
\begin{equation}
\{M\in \mathcal{S}, |M| < \rho\}
\end{equation}
is compact in $S$, $D^2F(r^{\kappa-2}D^2(M+h_r))$ is uniformly bounded. When $\gamma \geq 0$, $\kappa \geq 2$, and thus
\begin{equation}
\|\D^2_{M} L(M,p)\| \leq r^{\kappa-2}\|D^2F(r^{\kappa-2}M + r^{\kappa-2}D^2h_r)\|
\end{equation} 
is uniformly bounded. 

For $\gamma < 0$, it follows from assumption (H3) by choosing $t = r^{\kappa-2}$ that
\begin{equation}
\|\D^2_{M} L(M,p)\| \leq \|r^{\kappa-2}D^2F(r^{\kappa-2}M + r^{\kappa-2}D^2h_r)\| \leq K.
\end{equation}

Now 
\begin{equation}
|D^2_{p}L(M,p)| = |\gamma(\gamma-1)(p+h_r)^{\gamma-2}|
\end{equation}
Since $h_r \sim 1$ and thus for $|p| < \delta$, the above is uniformly bounded.

Therefore, we have shown that $\|D^2L\|$ is uniformly bounded in a neighborhood of $0$. Now we can apply Savin's theorem in \cite{savin}, which yields

\begin{equation}
\|u_r-h_r\|_{C^{2,\sigma}(B_{1/2})} \leq C.
\end{equation}
Scaling back, we have
\begin{equation}
\|u-h\|_{C^{2,\sigma}(B_{r/2}(x_0))} \leq Cr^{\kappa -2-\sigma} \leq C(\delta)
\end{equation}
as $r > \delta/2$, and the result of the theorem follows.
\end{proof}

\begin{cor} \label{ukestimate}
Let $u$ be a viscosity solution to \eqref{problem}, then $u\in C^\infty(B_{1}\cap\{x_n - q(x') > \delta\})$. Moreover, there exists a positive constant $C_k$ such that for any $x \in B_{1}\cap\{x_n - q(x') > \delta\}\}$, 
\begin{equation}
|D^k u| \leq C_k (x_n- q(x'))^{\frac{2}{1-\gamma}-k}
\end{equation}
\end{cor}
\begin{proof}
It follows from Theorem \ref{Savin} that solutions are $C^{2,\sigma}$ in $B_{1}\cap\{x_n - q(x') >\delta\}$. Since $F \in C^\infty(S)$, it follows from Bernstein's technique, for example from Proposition 9.1 \cite{caffarellicabre}, that $u \in C^\infty(B_{1}\cap\{x_n - q(x') > \delta\})$. 

For any $x \in B_{1}\cap\{x_n - q(x') > \delta\})$, $r$ small such that $B_r(x) \subset \subset B_{1}\cap\{x_n - q(x') > \delta\})$. Let 
\begin{equation}
u_r = r^{-\kappa} u(x + ry),
\end{equation}
then $u_r \in C^\infty(B_1)$ and 
\begin{equation}
\|u_r\|_{C^k(B_{1}(x))} \leq C.
\end{equation}
Scaling back yields the Corollary.
\end{proof}

\begin{lem}\label{uderestimate}
For any $x \in B_{1/2}(x_0',q(x_0'))\cap\{x_{0,n} - q(x_0') >\delta\}$, we have
\begin{align} \label{uestimate}
&(1).u(x',x_n+q(x')) \approx h(x_n)\\
&(2).u_n(x',x_n+q(x')) \approx h(x_n)/x_n\\
&(3).|\nabla u(x',x_n+q(x'))| \leq Ch(x_n)/x_n\\
&(4). |D^2u(x',x_n+q(x'))| \leq Ch(x_n)/x_n^2\\
&(5). |\nabla \frac{u_e}{u_n}(x',x_n+q(x'))| \leq C/x_n.
\end{align}
\end{lem}

\begin{proof}
We prove these properties with $x_0' = 0$ and $q(x_{0}') = 0$. The general case follows from translating and rescaling. 

By \eqref{flat}, we have
\begin{align}
h(x_n - \eps) \leq u(x', x_n) \leq h(x_n + \eps).
\end{align}
Recall that $x _n > \delta = C_0\eps$, by Proposition \ref{hprop} we have
\begin{align}
&u(x',q(x') + x_n) \leq h(x_n + \eps)  \leq C h(x_n),\\
&u(x',q(x') + x_n) \geq h(x_n - \eps) \geq c h(x_n).
\end{align}
This proves \eqref{uestimate} (1). 

We first note that (3), (4) follows from Corollary \ref{ukestimate} by taking $k =1,2$. 

For (2), let $0< \delta <1 $ small, for $t \in [-\delta,0]$, let $v(t) = u(x',x_n+q(x') + tx_n)$, then
\begin{align}
v(t) &\approx h((1+t)x_n)\\
v'(t) &= u_n(x',x_n+q(x') + tx_n) x_n\\
v''(t) &= u_{nn}(x',x_n+q(x') + tx_n) x_n^2.\\
\end{align}
We note that by (4), for $|t|$ small,
\begin{align}
v''(t) &= u_{nn}(x',x_n+q(x') + tx_n) x_n^2 \leq C\frac{h((1+t)x_n)}{(1+t)^2x_n^2} x_n^2\\
&\leq C\frac{(1+t)^{\frac{2}{1-\gamma}}x_n^{\frac{2}{1-\gamma}}}{(1+t)^2} \leq C(1+t)^{\frac{2}{1-\gamma}-2}h(x_n) \leq Ch(x_n)
\end{align}
We have
\begin{equation}
v'(0) - v'(t) = \int_{t}^0 v''(s)ds.
\end{equation}
Integrating both sides from $-\delta$ to $0$ and rearranging gives
\begin{align}
v'(0) &= \frac{v(0) - v(-\delta)}{\delta} + \frac{1}{\delta}\int_{-\delta}^0\int_{t}^0 v''(s)dsdt.
\end{align}
For the second term, we have
\begin{equation}
\frac{1}{\delta}\int_{-\delta}^0\int_{t}^0 v''(s)dsdt \leq \frac{1}{\delta}\int_{-\delta}^0 |t|Ch(x_n)dt \leq C\delta h(x_n).
\end{equation}
For the $v(0) - v(-\delta)$ term, by mean value theorem, there exists $\delta' \in [0,\delta]$ such that $v(-\delta') \delta = v(0) - v(-\delta)$. By (1), we have
\begin{align}
\frac{v(0) - v(-\delta)}{\delta} = v(-\delta') \approx h((1-\delta')x_n) \approx h(x_n).
\end{align}
Combining the above, we have
\begin{align}
C(h(x_n) - \delta h(x_n)) \leq v'(0) \leq C(h(x_n) + \delta h(x_n)).
\end{align}
Recall that $v'(0) = u_n(x',x_n+q(x'))x_n$, so choosing $\delta$ small gives (2).

To obtain (5), we differentiate and using (2)-(4) to get
\begin{equation}
\left|\nabla \frac{u_e}{u_n}(x',x_n+q(x'))\right| = \left|\frac{u_n\nabla u_e - u_e \nabla u_n}{u_n^2}\right| \leq \frac{C}{x_n}.
\end{equation}
\end{proof}

\section{Hodograph Transform} \label{Sec:Hodograph}
By the result of the previous section, $u$ is $C^{\infty}$ away from the free boundary. We perform a partial Hodograph transform $\Phi$ with respect to the 1D solution $h$, $\Phi: B_1\cap (\Omega^+\cup \Gamma) \rightarrow \RR^{n,+} \subset \RR^n$, where $\RR^{n,+}$ is the upper half plane, by choosing $v$ to be the number such that
\begin{align} \label{hodo}
u(y',v(y',y_n)) = h(y_n).
\end{align}
We note that near the boundary $\{y_n \leq \delta\}$, $v$ is potentially multivalued. In all of the following, the $u$ derivatives are evaluated at $(x',x_n) = (y',v(y',y_n))$ and $v$ derivatives are evaluated at $(y', y_n)$. Differentiating \eqref{hodo} yields
\begin{align}
&u_nv_n = h'(y_n),\\
&u_i+u_nv_i = 0.
\end{align}
This gives
\begin{align}
&u_n = \frac{h'}{v_n},\\
&u_i = -\frac{h'}{v_n}v_i.
\end{align}
We continue to differentiate \eqref{hodo} and get
\begin{align}
&u_{nn}v_n^2 + u_nv_{nn} = h''\\
&u_{in}v_n + u_{nn}v_iv_n + u_nv_{in}=0\\
&u_{ij}+u_{in}v_j+u_{nj}v_i + u_{nn}v_iv_j + u_nv_{ij}=0
\end{align}
This yields
\begin{align} \label{D^2uinv}
u_{nn} &=  \frac{h''}{v^2_n} -\frac{h'}{v^3_n}v_{nn},\\
u_{in} &= -\frac{h'}{v^2_n}v_{in} - \frac{h''}{v_n^2}v_i +\frac{h'}{v_n^3}v_iv_{nn},\\
u_{ij} &=-\frac{h'}{v_n}v_{ij} + \frac{h'}{v_n^2}v_iv_{nj} +\frac{h'}{v^2_n}v_jv_{in} - \frac{h'}{v_n^3}v_iv_jv_{nn} +\frac{h''}{v_n^2}v_iv_j.
\end{align}

Let $G(D^2v, Dv) = D^2u$. Then $v$ satisfies the PDE (in the sense of definition \ref{multivisdefn2}),
\begin{equation}\label{vpde}
F(G(D^2v,Dv)) = h^\gamma(y_n).
\end{equation}

Let $1 \leq e <n$,
\begin{equation}
w = v_e.
\end{equation}
Then $w$ satisfy the following linear PDE.
\begin{equation}\label{wpde1}
\sum_{i,j=1}^n\sum_{k,l = 1}^n\frac{\D F}{\D r^{kl}}\left(\frac{\D G^{kl}(r,p)}{\D r^{ij}} w_{ij} + \frac{\D G^{kl}(r,p)}{\D p^i} w_i\right) = 0.
\end{equation}
Here the derivative of $G$ terms can be read from \eqref{D^2uinv}. Explicitly, for $1\leq k,l,i,j < n$, we have\\
\begin{align}\label{coef1}
& \frac{\D G^{nn}}{	\D r^{nn}} = -\frac{h'}{v_n^3}, && \frac{\D G^{kn}}{\D r^{nn}} = \frac{h'}{v_n^3}v_k, &&\frac{\D G^{kl}}{\D r^{nn}} = -\frac{h'}{v_n^3} v_kv_l,\\
& \frac{\D G^{nn}}{	\D r^{in}} = 0, && \frac{\D G^{in}}{\D r^{in}} = -\frac{h'}{v_n^2}, &&\frac{\D G^{ij}}{\D r^{in}} = \frac{h'}{v_n^2}v_j,\\
& \frac{\D G^{nn}}{	\D r^{ij}} = 0, && \frac{\D G^{in}}{\D r^{ij}} = 0, &&\frac{\D G^{ij}}{\D r^{ij}} = -\frac{h'}{v_n},
\end{align}

\begin{align}\label{coef2}
&\frac{\D G^{nn}}{\D p^n} = -2\frac{h''}{v_n^3}+ 3\frac{h'}{v_n^4}v_{nn},\\
&\frac{\D G^{kn}}{\D p^n} = 2\frac{h'}{v_n^3}v_{kn}+ 2\frac{h''}{v_n^3}v_{k} - 3 \frac{h'}{v_n^4}v_kv_{nn},\\
&\frac{\D G^{kl}}{\D p^n} = \frac{h'}{v_n^2}v_{kl} - 2\frac{h'}{v_n^3}v_kv_{nl}- 2\frac{h'}{v_n^3}v_lv_{kn} + 3\frac{h'}{v_n^4}v_kv_lv_{nn} -\frac{h''}{v_n^2}v_kv_l,\\
&\frac{\D G^{nn}}{\D p^i} = 0\\
&\frac{\D G^{in}}{\D p^i} = -\frac{h''}{v_n^2} + \frac{h'}{v_n^3}v_{nn},\\
&\frac{\D G^{il}}{\D p^i} = \frac{h'}{v_n^2}v_{nl}- \frac{h'}{v_n^3}v_lv_{nn}+ \frac{h''}{v_n}v_l.
\end{align}

We next establishes some preliminary properties of $v$, which follows from the properties of $u$.
\begin{lem}\label{xnyn}
In $B_{1/2} \cap \{y_n > \delta\}$, there exists a constant $C$ such that
\begin{equation}
C^{-1}y_n \leq v(y',y_n) - q(y') \leq Cy_n.
\end{equation}
\end{lem}
\begin{proof}
By Lemma \ref{uderestimate} (1), for $x_n = v(y',y_n) - q(y') > \delta$,
\begin{equation}
C^{-1}h(v(y',y_n) - q(y')) \leq u(y', v(y',y_n)) \leq Ch(v(y',y_n)-q(y')).
\end{equation}
By definition of the hodograph transform, $u(y', v(y',y_n)) = h(y_n)$, and thus
\begin{equation}
C^{-1}h(v(y',y_n)-q(y')) \leq h(y_n) \leq Ch(v(y',y_n)-q(y')).
\end{equation}
Since $h$ is strictly monotone, applying $h^{-1}$ gives
\begin{equation}
C^{-1}(v(y',y_n)-q(y')) \leq y_n \leq C(v(y',y_n)-q(y')),
\end{equation}
which proves the lemma.
\end{proof}
\begin{rem}
By the above Lemma \ref{xnyn}, the hodograph variable $y_n$ and the distance 
to the free boundary $x_n - q(x')$ in the original coordinates are comparable. In the rest of this section, we work in the hodograph variable and write $x$ for 
$y$. In all following estimates of the form $C/x_n$, $h(x_n)/x_n$, etc. we used the above Lemma implicitly to switch to hodograph variables.
\end{rem}

In the rest of the section, we will always be in the hodograph variable and denote it by $(x',x_n)$. 

\begin{lem}
In $B_{1/2}\cap \{x_n > \delta\}$, we have
\begin{align}
&(1). |Dv| \leq C.\\
&(2). |D^2v| \leq \frac{C}{x_n}.\\
\end{align}
\end{lem}
\begin{proof}
Let $1 \leq i,j \leq n-1$, for the first derivatives of $v$, we have
By Lemma \ref{hprop} and \ref{uderestimate},
\begin{align}
&|v_i| \leq \left|\frac{u_i}{u_n}\right| \leq C,
&C ^{-1} \leq |v_n| \leq \left|\frac{h'}{u_n}\right| \leq C.
\end{align}
For the second derivatives of $v$, we have
\begin{align}
&|v_{nn}| \leq \left|\frac{h''}{u_n} + \frac{u_{nn}v_n^2}{u_n}\right| \leq \frac{1}{x_n},\\
&|v_{in}| \leq \left|\frac{u_{in}v_n}{u_n} + \frac{u_{nn}v_iv_n}{u_n}\right| \leq \frac{1}{x_n}\\
&|v_{ij}| \leq \left|\frac{u_{ij}}{u_n} + \frac{u_{in}v_j}{u_n} + \frac{u_{nj}v_j}{u_n} + \frac{u_{nn}v_iv_j}{u_n} \right| \leq \frac{1}{x_n}.
\end{align}
\end{proof}
It follows immediately that
\begin{equation}
\|w\|_{L^\infty(B_{1/2}\cap \{x_n > \delta\})} = \|v_e\|_{L^\infty(B_{1/2}\cap \{x_n > \delta\})} \leq C.
\end{equation}

We finally state the definition for viscosity solutions which are compatible with  multivalued functions $v$ on $B_1\cap \{x_n >0\}$. Note that the map $G$ is order reversing, and thus $-F\circ G$ is elliptic.

\begin{defn}\label{multivisdefn2}
Let $v: B_1^+\rightarrow \RR$ be a possibly multivalued function.  We say $v$ is a \textit{viscosity solution} of \eqref{vpde} if the followings hold.

\begin{enumerate}
    \item The graph of $v$ is closed.

    \item For any 
    $x_0 \in B_1 \cap \{x_n > 0\}$ and $\phi \in C^2$ touching $v$ 
    from above (resp.\ below) at $x_0$, i.e.\ $\phi(x_0) \in v(x_0)$ 
    and $\phi \geq v$ (resp.$\phi \leq v$) in a neighborhood of $x_0$, then
    \begin{equation}
       F(G(D^2\phi(x_0),D\phi(x_0))) - h^\gamma(x_{0,n}) \leq 0. \quad (\geq 0\text{ resp. })
    \end{equation}
\end{enumerate}
\end{defn}

In the above definition, $\phi \geq v$ means $\phi$ is greater than or equal to all possible values of $v$. We note that $\partial G^{kl}/\partial r_{ij}$ is negative definite, so that $F \circ G$ is decreasing in $D^2v$, in analogy with $-\Delta$ and the inequalities in Definition \ref{multivisdefn2} are consistent with  the usual definition for viscosity solutions.

\section{Harnack Inequality}\label{Sec:Harnack}

This section proves the following Harnack type inequality for the equation \eqref{vpde} in $\{x_n > 0\}$.

\begin{thm}\label{Harnack}
    There exists a universal constant $\eps_0$ such that if $v$ is a viscosity solution to \eqref{vpde} in $B_1\cap \{x_n > 0\}$, and it satisfies for some $x_0 \in \{x_n = 0\}$,
    \begin{equation}\label{epsrflat}
        x_n + a_0 \leq v(x) \leq x_n + b_0 \text{ in } B_r^+(x_0) \subset B_1(x_0)
    \end{equation}
    with $b_0-a_0 < \eps r$, $ \eps < \eps_0$, then 
    \begin{equation}
        x_n + a_1 \leq v(x) \leq x_n + b_1 \text{ in } B^+_{r/20}(x_0) 
    \end{equation}
    with $a_0 \leq a_1 \leq b_1\leq b_0$, $b_1 - a_1 \leq (1-c)\eps r$, and $0 < c<1$ universal.
\end{thm}

By the flatness assumption, $v$ satisfies \eqref{epsrflat} with $r=1$, and we can apply the above theorem repeatedly and obtain
\begin{equation}
     x_n + a_m \leq v(x) \leq x_n + b_m \text{ in } B^+_{20^{-m}}(x_0)
\end{equation}
with $b_m - a_m \leq (1-c)^m \eps$, as long as 
\begin{equation}
    (1-c)^m20^m\eps \leq \eps_0.
\end{equation}
For all the $m$ such that the above holds, the oscillation of the function
\begin{equation}
    \tilde{v}_\eps(x) = \frac{v(x)-x_n}{\eps}
\end{equation}
is less than $(1-c)^m = 20^{-m\alpha_1}$ in $B^+_{20^{-m}}(x_0)$, and thus the following Corollary holds.
\begin{cor}\label{harnackcor}
Let $v$ be a solution to \eqref{vpde} satisfying \eqref{epsrflat} with $r=1$. Then in $B^+_1(x_0)$, for $|x-x_0| \geq \eps/\eps_0$, we have
\begin{equation}
    |\tilde{v}_\eps(x) - \tilde{v}_\eps(x_0)| \leq C|x-x_0|^{\alpha_1}.
\end{equation}
\end{cor}

Define the following cubes in $\{x_n >0\}$,
\begin{align}
&Q_{s,t}(x_0) = \{(x',x_n): |x'-x'_0| < s, |x_n-x_{0,n}| < t\},\\
&Q_{s}'(x_0') = \{x' \in \RR^{n-1}: |x'-x'_0| < s\}
\end{align}

The proof of the Harnack inequality follows from the following Lemma. 
\begin{lem}
    There is a universal constant $\eps_0$ such that if $v$ is a viscosity solution to \eqref{vpde} in $Q_{3/4,1}$, and it satisfies,
    \begin{equation}
        p(x) \leq v(x) \leq p(x) + \eps, x\in Q_{3/4,1}, p(x) = x_n + \sigma, |\sigma| < 1/10,
    \end{equation}
    then if at $x_0 = 1/5 e_n$, 
    \begin{equation}
        v(x_0) \geq p(x_0) + \eps/2,
    \end{equation}
    then for some $0<c<1$,
    \begin{equation}
    v \geq p+c\eps \in Q_{1/2,1}.
    \end{equation}
    
    Analogously, if at $x_0 = 1/5 e_n$, 
    \begin{equation}
        v(x_0) \leq p(x_0) + \eps/2,
    \end{equation}
    then for some $0<c<1$,
    \begin{equation}
    v \leq p + (1-c)\eps \in Q_{1/2,1}.
    \end{equation}
\end{lem}
\begin{proof}
We prove the first statement. We consider the strip
\begin{equation}
Q = Q'_{3/4}(0) \times (0,1)
\end{equation}
Fix $\delta < 1$ to be chosen later. In the region $Q'_{3/4}(x_0') \times (\delta,1)$, we note that $v - p$ satisfies a uniformly elliptic equation. Since $v$ and $x_n$ are both solutions to \eqref{vpde},
\begin{equation}
F(G(D^2v,Dv)) - F(G(D^2p,Dp)) = h^\gamma(x_n) - h^\gamma(x_n) = 0,
\end{equation}
and the left hand side can be rewritten as
\begin{equation}
F(G(D^2v,Dv)) - F(G(D^2p,Dp)) = \int_0^1 \frac{d}{dt}F(G(D^2(tv + (1-t)p),D(tv + (1-t)p)))dt.
\end{equation}
Therefore, Harnack inequality implies
\begin{equation} \label{intharnack}
v(x) - p(x) \geq c_0\eps \in \overline{Q'_{3/4}}(x_0') \times [\delta,1].
\end{equation}
Consider the barrier
\begin{align}
&v_t := p(x) + c_0\eps(\psi(x')w(x_n)-1)+t, \\
&w_\eta(x_n) := \frac{x_n^\beta - \eta^\beta}{\delta^\beta - \eta^\beta}.
\end{align}
where $\beta <0$, $\psi < 1$ is a smooth function such that $|D^2_{x'} \psi| < C$, $\psi = 0$ on $\D Q'(3/4)$ and $\psi > 0$ in $Q'(3/4)$. We note that 
\begin{align}
&w_\eta(\delta) = 1,\qquad
w_\eta(x_n) \rightarrow -\infty \text{ as } x_n \rightarrow 0\\
&w'_\eta(x_n)=\frac{\beta x_n^{\beta-1}}{\delta^\beta - \eta^\beta},\qquad
w''_\eta(x_n)=\frac{\beta(\beta-1) x_n^{\beta-2}}{\delta^\beta - \eta^\beta},\\
&w_\eta(x_n) = \frac{(x_n/\eta)^\beta-1}{(\delta/\eta)^\beta-1} \rightarrow 1, \text{ for any fixed }x_n, \text{ as }\eta \rightarrow 0.
\end{align} 

We first show that $v_t$ is a subsolution to \eqref{vpde}. By \eqref{D^2uinv}, we have
\begin{align} \label{Gvt}
G^{nn}(D^2v_t,Dv_t) &=  \frac{h''}{v^2_{t,n}} -\frac{h'}{v^3_{t,n}}v_{t,nn},\\
G^{in}(D^2v_t,Dv_t) &= -\frac{h'}{v^2_{t,n}}v_{t,in} - \frac{h''}{v_{t,n}^2}v_{t,i} +\frac{h'}{v_{t,n}^3}v_{t,i}v_{t,nn},\\
G^{ij}(D^2v_t,Dv_t) &=-\frac{h'}{v_{t,n}}v_{t,ij} + \frac{h'}{v_{t,n}^2}v_{t,i}v_{nj} +\frac{h'}{v^2_{t,n}}v_{t,j}v_{t,in} - \frac{h'}{v_{t,n}^3}v_{t,i}v_{t,j}v_{t,nn} +\frac{h''}{v_{t,n}^2}v_{t,i}v_{t,j},
\end{align}
and 
\begin{align}
&v_{t,i} = c_0\eps\psi_iw,\\
&v_{t,n} = 1+c_0\eps\psi w' ,\\
&v_{t,ij} = c_0\eps\psi_{ij}w ,\\
&v_{t,in} = c_0\eps\psi_i w',\\
&v_{t,nn} = c_0\eps \psi w''.
\end{align}
Since $h$ is a 1D solution to $F(D^2h) = h^\gamma$, we have
\begin{align}\label{ftcfg}
F(G(D^2v_t,Dv_t)) - h^\gamma &= F(G(D^2v_t,Dv_t)) - F(h'')\\
&= \int_0^1 DF(sG(D^2v_t,Dv_t) + (1-s)h''): [G(D^2v_t,Dv_t) - h'']ds.
\end{align}
and we have
\begin{align}
G^{ij}(D^2v_t,Dv_t) &= -h'c_0\eps\psi_{ij}w + O(\eps^2),\\
G^{in}(D^2v_t,Dv_t) &=-h'c_0\eps\psi_iw'-h''c_0\eps\psi_i w + O(\eps^2),\\
G^{nn}(D^2v_t,Dv_t) - h'' & = -2h''c_0\eps\psi w' - h'c_0\eps \psi w'' + O(\eps^2).
\end{align}
We claim that $G^{nn}(D^2v_t,Dv_t) - h'' \leq 0$. Note that $c_0,\eps,\psi, h' \geq 0$, so it is sufficient to have
\begin{equation}
-2\frac{h''}{h'}w'-w'' \leq 0.
\end{equation}
By Proposition \ref{hconstant}, $\frac{x_n h''}{h'} = \frac{1+\gamma}{1-\gamma} + o(1)$ as $x_n \rightarrow 0$. Therefore, we need,
\begin{align}
-2\frac{h''}{h'}w'-w'' = -2\frac{1+\gamma}{1-\gamma}c_1\beta x_n^{\beta-2} - \beta(\beta-1)c_1x_n^{\beta-2} +o(1) \leq 0.
\end{align}
for $c_1 = \frac{1}{\delta^\beta-\eta^\beta} < 0$.
Dividing by $-c_1 >0$, $\beta < 0$, and $x_n^{\beta-2} >0$ gives
\begin{equation}
2\frac{1+\gamma}{1-\gamma}+ \beta -1  \geq 0,
\end{equation}
Which is satisfied for 
\begin{equation}
\frac{-1-3\gamma}{1-\gamma}\leq \beta < 0
\end{equation}
which holds for some $\beta$ if $\gamma > -1/3$, and thus the claim is true for $x_n$ small.

Recall that 
\begin{equation}
h^\gamma \sim x_n^{\frac{2}{1-\gamma}-2}.
\end{equation}
Therefore,
\begin{align}
|G^{ij}(D^2v_t,Dv_t)| &\leq C \eps x_n^{\frac{2}{1-\gamma}+\beta-1}+ O(\eps^2),\\
|G^{in}(D^2v_t,Dv_t)| &\leq C \eps x_n^{\frac{2}{1-\gamma}+\beta-2}+ O(\eps^2),\\
|G^{nn}(D^2v_t,Dv_t) - h''|&\leq C \eps x_n^{\frac{2}{1-\gamma}+\beta-3} + O(\eps^2).
\end{align}
Since $\delta \in (0,1)$, we choose $\delta$ so that for $x_n < \delta$ small, 
\begin{align}
DF(sG(D^2v_t,Dv_t) &+ (1-s)h'') : [G(D^2v_t,Dv_t) - h'']  \\ &\leq CD_{nn}F(sG(D^2v_t,Dv_t) + (1-s)h'')[G^{nn}(D^2v_t,Dv_t) - h''] < 0
\end{align}
and thus $v_t$ is a subsolution. 

By definition, 
\begin{align}
&v_{0} = p(x) - c_0\eps < p(x) &&\text{ on } \D Q'_{3/4}(x_0) \times (0,\delta),\\
&v_{0} =  p(x) + c_0\eps(\psi(x')-1)  < p(x) + c_0\eps &&\text{ on }  Q'_{3/4}(x_0) \times \{\delta\},\\
&v_{0} \leq  p(x) - \frac{1}{2} c_0\eps < p(x) &&\text{ on }  Q'_{3/4}(x_0) \times \{0\},
\end{align}
which gives $v_{0} < v$ on $\D Q$ and maximum principle implies
\begin{equation}
v_0 < v \text{ in } Q.
\end{equation}

Now let 
\begin{equation}
t_0 = \sup\{t>0 : v_t < v \text{ in } Q'(x_0)\times(0,\delta) \} > 0.
\end{equation}
We wish to show that $t_0 \geq c_0\eps$. Then we have
\begin{equation}
v(x) \geq v_{t_0}(x) \geq p(x) + c_0\eps\psi(x')w.
\end{equation}
For any $x \in \overline{Q'}_{1/2}(x_0) \times [0,\delta]$, we can pick $\eta$ small such that $\psi(x')w(x_n) > c >0$, we conclude that 
\begin{equation}
v(x) - p(x)  \geq c\eps \text{ on } \overline{Q'}_{1/2}(x_0) \times [0,\delta].
\end{equation}

Suppose for contradiction that $t_0 < c_0\eps$ at some $\overline{x} \in Q'(x_0)\times(0,\delta)$, we have
\begin{equation} \label{vt0=v}
v_{t_0}(\overline{x}) = v(\overline{x}).
\end{equation}

By definition, 
\begin{align}
&v_{t_0} = p(x) - c_0\eps + t_0 < p(x) &&\text{ on } \D Q'_{3/4}(x_0) \times (0,\delta),\\
&v_{t_0} =  p(x) + c_0\eps(\psi(x')-1) + t_0 < p(x) + c_0\eps &&\text{ on }  Q'_{3/4}(x_0) \times \{\delta\},
\end{align}
By \eqref{intharnack}, 
\begin{equation}
v_{t_0} < v \text{ on }  Q'_{3/4}(x_0) \times \{\delta\}.
\end{equation}
Finally we note that on $Q'_{3/4}(x_0) \times \{0\}$, 
\begin{equation}
v_{t_0} < v \text{ on } Q'_{3/4}(x_0) \times \{0\}.
\end{equation}

\begin{comment}

By \eqref{nondegeneracy}, we have for any $x_n > 0$, 
\begin{equation}
h(x_n) - h(0) =  u(x',v(x',x_n)) - u(x',v(x',0)) \geq c|v(x',x_n) - v(x',0)|^{\frac{2}{1-\gamma}}.
\end{equation}
Rewriting gives
\begin{equation}
|v(x',x_n) - v(x',0)| \leq Ch(x_n)^{\frac{1-\gamma}{2}} \leq C x_n
\end{equation}
Suppose by contradiction that $\overline{x}_n = 0$,
\begin{equation}
v_{t_0}(\overline{x}', 0) = v(\overline{x}',0)
\end{equation}
We have,
\begin{align}
v_{t_0}(\overline{x}',x_n) - v(\overline{x}',x_n) &= \left(v_{t_0}(\overline{x}',x_n)\right) - \left(v_{t_0}(\overline{x}',0) + v(\overline{x}',x_n) - v(\overline{x}',0)\right)\\
& \geq x_n + c_0\eps\psi(\overline{x}')c_1 x_n^\beta - Cx_n\\
& > 0
\end{align}
for $x_n$ small since $\beta <1$, which is a contradiction to the definition of $t_0$. 
\end{comment}
Therefore we have shown
\begin{equation}
v_{t_0} < v \text{ on } \D Q,
\end{equation}
and using maximum principle again implies
\begin{equation}
v_{t_0} < v \text{ in } Q,
\end{equation}
which is a contradiction to \eqref{vt0=v} and concluded the proof of the first statement in the Lemma. The second statement follows analogously by replacing the subsolution with supersolution 
\begin{equation}
v_t := p(x) - c_0\eps(\psi(x')w(x_n)-1)-t.
\end{equation}
\end{proof}

\begin{proof}[Proof of Theorem \eqref{Harnack}]
Without loss of generality, suppose $x_0 = 0$, define $v_r(x): B_2^+(x) \rightarrow \RR$,
\begin{equation}
v_r(x) = \frac{2}{r}v(\frac{1}{2}rx).
\end{equation}
By \eqref{epsrflat}, we have
\begin{equation}
x_n - 2\eps \leq v_r(x) \leq x_n + 2\eps.
\end{equation}
We note that $Q_{3/4,1}(0) \subset B_2^+(0)$, therefore the previous Lemma applies and yields the Theorem.
\end{proof}

\section{Analysis of the Linear Equation}\label{Sec:Liouville}

This section focuses on the analysis of the linear equation \eqref{wpde1}. First we note that after multiplying equation \eqref{wpde1} by $-h'^{-1}$, the equation can be written as (where we used $u$ for $w$),
\begin{equation} \label{nondivL}
Lu = \sum_{i,j = 1}^{n}a^{ij}(x)u_{ij} + \frac{\alpha}{x_n} b^i(x) u_i=0 
\end{equation}
where $\lambda I \leq (a^{ij}(x)) \leq  \Lambda I$, $-C \leq b^i(x) \leq C$ are bounded above and below by positive constants and the constants $\lambda, \Lambda, C$ are independent of $x$.

Consider the frozen coefficient equation
\begin{equation}\label{frozen}
{L}_0u = \sum_{i,j = 1}^{n}a^{ij}_0u_{ij} + \sum_{k= 1}^{n-1}\frac{\alpha}{|x_n|} b^k_0 u_k +  \frac{\alpha}{x_n} u_n = 0
\end{equation}
where $a^{ij}_0$, $b^i_0$ are given by \eqref{coef1}, \eqref{coef2}, evaluating at $v = x_n$.

We can take without loss of generality that $a^{nn}_0 = b^n_0 = 1$. Multiplying through by $|x_n|^\alpha$, \eqref{frozen} can also be written as in divergence form
\begin{equation} \label{divL}
{\mathcal{L}}_0u := (|x_n|^\alpha u_n)_n + |x_n|^\alpha\left(\sum_{k,l=1}^{n-1}{a}^{kl}_0u_{kl} + 2{a}^{kn}_0u_{kn} + \frac{\alpha}{|x_n|}{b}^k_0 u_k\right)  = 0 
\end{equation}
where
\begin{equation}\label{alpha}
\alpha = \frac{2(1+\gamma)}{1-\gamma}.
\end{equation}

\begin{rem}
We observe that in \eqref{wpde1}, after freezing the coefficients at $v = x_n$, the normal direction can be rewritten in divergence form since
\begin{equation}
((h')^2u_n)_n = 2h'h''u_n + (h')^2u_{nn}.
\end{equation}
We take $\alpha$ to be the constant satisfying $(h')^2 \sim x_n^\alpha$, hence the model equation \eqref{frozen} and \eqref{divL}.
\end{rem}

The above equations are defined in all $\{x_n \neq 0\}$ whereas the original \eqref{wpde1} is only satisfied in $\{x_n > 0$\}. The next theorem shows next that appropriate solutions of \eqref{divL} in $\{x_n >0\}$ can be extended by even extension and solve the corresponding equation.

\begin{thm}\label{evenextension}
Suppose $u$ is a viscosity solution of $\mathcal{L}_0u = 0$ in $\{x_n >0\}$. Assume that $u$ is continuous up to $\{x_n = 0\}$. Then, for $\gamma > -1/3$, we can extend $u$ by even extension and the resulting function still denoted by $u$ is a viscosity solution to the equation $\mathcal{L}_0u = 0$.
\end{thm}
\begin{proof}
We first check that for $\gamma > -1/3$, $u$ is a viscosity solution on $\{x_n = 0\}$. Suppose $\phi \in C^2$ touches $u$ at $(x_0',0)$ from above. Then 
\begin{align}
\mathcal{L}_0\phi = (|x_n|^\alpha \phi_n)_n + |x_n|^\alpha\left(\sum_{k,l=1}^{n-1}{a}^{kl}_0\phi_{kl} + 2{a}^{kn}_0\phi_{kn} + \frac{\alpha}{|x_n|}{b}^k_0 \phi_k\right)
\end{align}
We proceed to show all of the above terms are $0$ at $x_n = 0$. Recall that 
\begin{equation}
    \alpha = \frac{2(1+\gamma)}{1-\gamma} > 1,
    \quad \text{when} \quad \gamma > -\frac{1}{3}.
\end{equation}
Therefore as $x_n \rightarrow 0$,
\begin{align}
&\left|(x_n^\alpha{a^{nn}}_0\phi_n)_n\right| \leq C\left(|D\phi||x_n|^{\alpha-1} + |D^2\phi||x_n|^\alpha\right) \rightarrow 0.\\
&\left|x_n^\alpha{a}^{kl}_0\phi_{kl}\right| \leq C|D^2\phi||x_n|^\alpha\rightarrow 0,\\
&\left|x_n^\alpha{a}^{kn}_0\phi_{kn}\right| \leq C|D^2\phi||x_n|^\alpha\rightarrow 0,\\
&\left|\frac{\alpha}{x_n} x_n^\alpha{b}^k_0 \phi_k\right| \leq C|D\phi||x_n|^{\alpha-1} \rightarrow 0.
\end{align}
Therefore, $\mathcal{L}_0\phi \geq 0$. The case where $\phi$ touches $u$ from below follows from the same argument.

Since $u$ is a viscosity solution in $\{x_n > 0\}$ and $u$ is even, $u$ is a viscosity solution in $\{x_n \neq 0\}$. By assumption $u$ is continuous on $\RR^n$, the Theorem is proved.
\end{proof}

\begin{rem}
The above extension can be understood as a Neumann boundary condition in the following sense. If $u$ is smooth enough, for $\gamma > -1/3$, 
\begin{equation}
|x_n|^\alpha u_n \rightarrow 0 \text{ as } x_n \rightarrow 0,
\end{equation}
and $u$ satisfies equation \eqref{divL} in the following sense.
\begin{align}
\int_{B_1(x',0)} |x_n|^\alpha u_n\phi_n = \int_{B_1(x',0)}(|x_n|^\alpha a^{ij}_0u_{ij} + |x_n|^\alpha a^{in}_0u_{in} + \alpha |x_n|^{\alpha -1}b^i_0u_{i})\phi
\end{align}
for $1 \leq i,j \leq n-1$. 
\end{rem}

\begin{rem}
We note that the equation $L_0 u = 0$ and $\mathcal{L}_0 u = 0$ is equivalent in $\{x_n \neq 0\}$. We henceforth always understand the two equations at $\{x_n = 0\}$ in the sense of Theorem \ref{evenextension}. Thus $L_0 u = 0$ and $\mathcal{L}_0 u = 0$ are equivalent in $\RR^n$ and we use them interchangeably.
\end{rem}

We next show that \eqref{nondivL} has a weak maximum principle. 

\begin{lem}\label{maxprin} 
Let $u$ be a viscosity solution of ${L}_0 u =0$ in $B_r$, then
\begin{equation}
    \sup_{B_r} {u} =  \sup_{\partial B_r} {u}.
\end{equation}
\end{lem}
\begin{proof}
Consider the auxiliary function 
\begin{equation}
v = u + \eps |x_n|^2.
\end{equation}
We first show that $v$ cannot attain its maximum in the interior of $B_r$. Suppose for contradiction that $v$ attains its max at some $x_0 \in B_r$, then
\begin{equation}
v(x) \leq v(x_0) = u(x_0) + \eps|x_{0,n}|^2.
\end{equation}
Which gives
\begin{equation}
u \leq u(x_0) +  \eps|x_{0,n}|^2 - \eps|x_n|^2.
\end{equation}
Therefore, $u(x_0) +  \eps|x_{0,n}|^2 - \eps|x_n|^2$ touches $u$ from above at $x_0$. We can calculate
\begin{align}
{L}_0(u(x_0) +  \eps|x_{0,n}|^2 - \eps|x_n|^2) &= -2\eps a^{nn}_0(\alpha+1) < 0.
\end{align}
which is a contradiction, and thus
\begin{equation}
\max_{B_r} v \leq \max_{\D B_r} v.
\end{equation}
Substituting gives
\begin{equation}
\max_{B_r} u + \eps |x_n|^2 \leq \max_{\D B_r} u + \eps |x_n|^2.
\end{equation}
Taking $\eps \rightarrow 0$ yields the Lemma.
\end{proof}

\begin{lem}\label{gradientbound}
Suppose $u$ is a viscosity solution to $L_0 u = 0$ in $B_1$, and $u \in C^2(\RR^n \setminus \{x_n = 0\})$. then 
\begin{equation}
\sup_{B_{1/2}}|D_{x'}u|\leq c\sup_{{B_1}}|u|.
\end{equation}
\end{lem}
\begin{proof}
We can differentiate the equation $L_0 u = 0$ in the tangential direction, so any tangential partial derivative $u_k$ satisfies $L_0 u_k = 0$. In the following, we sum over the index $1 \leq i,j \leq n$, $1 \leq k,l \leq n-1$.
From direct calculation
\begin{align}
L_0(|D_{x'}u|^2) &= 2a^{ij}_0u_{ik}u_{jk} + 2u_k(a^{ij}_0(u_{k})_{ij} + \frac{\alpha}{|x_n|}b^l_0(u_k)_l + \frac{\alpha}{x_n}(u_k)_n)\\
&= 2a^{ij}_0u_{ik}u_{jk} \geq c\lambda|D u_k|^2.
\end{align}
Choose a cutoff function $\phi = \eta^2$ for some $\eta \in C_0^1(B_r)$ with $\eta = 1$ in $B_{r/2}$, we have 
\begin{align}
L_0(\eta^2|D_{x'}u|^2) &= a^{ij}_0\left(2\eta_i\eta_j|D_{x'}u|^2 + 2\eta\eta_{ij}|D_{x'}u|^2 + 4\eta \eta_j u_ku_{ki} + 2\eta^2 u_{kj}u_{ki} + 2\eta^2 u_ku_{kij}\right) \\
&\quad + \frac{\alpha}{|x_n|}b^l_0 \left(2\eta\eta_i|D_{x'}u|^2 + 2\eta^2u_ku_{kl}\right)+ \frac{\alpha}{x_n} \left(2\eta\eta_n|D_{x'}u|^2 + 2\eta^2u_ku_{kn}\right)
\end{align}
We estimate the above terms in the following way.
\begin{align}
&2a^{ij}_0\eta_i\eta_j|D_{x'}u|^2 \geq 2\lambda |D\eta|^2|D_{x'}u|^2,\\
&4a^{ij}_0\eta \eta_j u_ku_{ki} \leq 4\Lambda(t|D\eta|^2|D_{x'}u|^2 + t^{-1}\eta^2 |D u_k|^2),\\
&2a^{ij}_0\eta^2u_{kj}u_{ki} \geq 2\eta^2\lambda|D u_k|^2.\\
&2a^{ij}_0\eta\eta_{ij}|D_{x'}u|^2 + \frac{\alpha}{|x_n|}b^l_02\eta\eta_l|D_{x'}u|^2 +\frac{\alpha}{x_n}2\eta\eta_n|D_{x'}u|^2 = |D_{x'}u|^2 2\eta L_0 \eta,\\
&2a^{ij}_0\eta^2 u_ku_{kij} + 2\frac{\alpha}{|x_n|}b^l_0\eta^2u_ku_{kl} + 2\frac{\alpha}{x_n}\eta^2u_ku_{kn}= 2u_k\eta^2L_0 u_k = 0.
\end{align}
Putting it together, we have
\begin{align}
L_0(\eta^2|D_{x'}u|^2)  &\geq (2\lambda|D\eta|^2 -2\eta L_0\eta - 4\Lambda t|D\eta|^2)|D_{x'}u|^2 + (2\eta^2\lambda - 4\Lambda t^{-1}\eta^2)|Du_k|^2 \geq -C|D_{x'}u|^2
\end{align}
for some positive constant $C$ depending on $\eta$ and some $t$ large. 
We note that
\begin{equation}
L_0 (u^2) = 2a^{ij}_0u_iu_j + 2u(a^{ij}_0u_{ij}+ \frac{\alpha}{|x_n|}b^l_0 u_l + \frac{\alpha}{x_n} u_n) = 2a^{ij}_0u_iu_j \geq \lambda|D u|^2
\end{equation}
Therefore, 
\begin{equation}
L_0(\eta^2 |D_{x'}u|^2 + s u^2) \geq 0
\end{equation}
for some $s$ large. Maximum principle implies 
\begin{align}
\sup_{B_r}(su^2 + \eta^2|D_{x'}u|^2) \leq \sup_{\D B_r} su^2
\end{align}
and the result of the lemma follows.

On $\{x_n = 0\}$, consider the tangential difference quotients
\begin{equation}
\tau_{x'} u = \frac{u(x+se_i) - u(x)}{|s|}.
\end{equation}
 Therefore, we can follow the computation above with $D_{x'}u$ replaced by $\tau_{x'} u$ and taking $s\rightarrow 0$ yields the Lemma.
\end{proof}

One immediate consequence is the following Liouville type theorem.
\begin{thm} \label{Liouville}
Suppose $u$ is a viscosity solution to $\mathcal{L}_0 u = 0$ in $\RR^n$. If $|u| \leq C|x|^m$ for some integer $m > 0$, then $u$ is a polynomial of degree $m$ in $x$.
\end{thm}
\begin{proof}
We first note that $u \in C^\infty(\RR^n \setminus \{x_n = 0\})$. By Lemma \ref{gradientbound} and a rescale to $B_R$, we have
\begin{equation}
\sup_{B_{R}}|D_{x'}u|\leq \frac{C}{R}\sup_{{B_{2R}}}|u|.
\end{equation}
We apply the above repeatedly and get
\begin{equation}
\sup_{B_{R}}|D^{m+1}_{x'}u|\leq \frac{C}{R^{m+1}}\sup_{{B_{2R}}}|u| \leq \frac{C}{R} \rightarrow 0 \text{ as }R\rightarrow \infty.
\end{equation}
Therefore $u$ is a polynomial of degree $m$ in $x'$, that is,
\begin{equation}\label{polysol}
u(x) = \sum_{|\mu|=0}^m c_\mu(x_n)x_1^{\mu_1}\cdots x_{n-1}^{\mu_{n-1}} 
\end{equation}
where $\mu$ is the multi-index $\mu = (\mu_1,\cdots,\mu_{n-1})$. 

Let $e_i = (0,\cdots, 1,0,\cdots,0)$ which has $1$ in the $i$-th entry and $0$ elsewhere. We can find the explicit growth in $x_n$ direction by plugging \eqref{polysol} into $\mathcal{L}_0 u = 0$, which yields the equation
\begin{align}
(|x_n|^\alpha c'_\mu(x_n)x_I^\mu)' + 2|x_n|^\alpha a^{in}_0\mu_i c'_\mu(x_n)x_I^{\mu_I - e_i}  +  c_\mu(x_n) \left(a^{ij}_0 |x_n|^\alpha \mu_i\mu_j x_{I}^{\mu_I - e_i -e_j} + \alpha |x_n|^{\alpha-1} b^i_0 \mu_i x_{I}^{\mu_{I - e_i}}\right) = 0.
\end{align}
for $K =  0,1,2,...,m$ denoting the highest powers in $x_\mu$. By matching the coefficient of the term of the same powers of $x_i$'s above, we get a system of equations. For $|\mu| = K = m$,
\begin{align}
(|x_n|^\alpha c'_\mu(x_n))' = 0.
\end{align}
which has solutions in 
\begin{equation}
c_\mu(x_n) = \operatorname{sgn}(x_n) C_1 x_n^{1-\alpha} + C_2.
\end{equation}
where $\operatorname{sgn}(x_n)$ is the sign of $x_n$. We note for $\gamma > -1/3$, $\alpha > 1$, and $x_n^{1-\alpha}$ is not bounded near $x_n = 0$, therefore $C_1 = 0$. 

For $|\mu| \leq m-1$, $c_\mu(x_n)$ satisfies an ODE of the form
\begin{equation} \label{generalode}
(|x_n|^\alpha c'_\mu(x_n))'  = g(x_n)
\end{equation}
where $g(x_n)$ is a linear combination of the powers of $x_n$, such that each of the term is of the form
\begin{equation}
|x_n|^\alpha c'_{\mu + e_i},\ |x_n|^{\alpha}c_{\mu + e_i + e_j},\ |x_n|^{\alpha -1}c_{\mu + e_i}.
\end{equation}
Moreover, solutions to \eqref{generalode} changes the order of $x_n$ on the right hand side by $+2-\alpha$, and yield the homogeneous solution $x_n^{1-\alpha}$ which has to have coefficient $0$ since this term is unbounded. Following this pattern, we can read the behavior of solutions inductively, 
\begin{align}
&|\mu| = m: && c_\mu \in \text{span} \{\cancel{\operatorname{sgn}(x_n)x_n^{1-\alpha}}, x_n^0\}\\
&|\mu| = m -1: && c_\mu \in \text{span} \{\cancel{\operatorname{sgn}(x_n)x_n^{1-\alpha}},x_n^{1}, x_n^0\},\\
&|\mu| = m -2: && c_\mu \in \text{span} \{\cancel{\operatorname{sgn}(x_n)x_n^{1-\alpha}}, x_n^{2}, x_n^{1}  , x_n^0\},\\
&\cdots \\
&|\mu| = 0: && c_\mu \in \text{span} \{\cancel{\operatorname{sgn}(x_n)x_n^{1-\alpha}}, x_n^{m},\cdots x_n^{1}, x_n^0\}.
\end{align}
and the result of the Theorem follows.
\end{proof}

Before proving a Schauder theorem, we state a technical Lemma that would be helpful in the proof.

\begin{lem}\label{Conormal}
Let $u$ be a $C^{2,\beta}$ solution to $L_0 u = f$ in $B_{1}(0)$, with $f \in C^{0,\beta}(B_1)$, then 
\begin{equation}
\sum_{i=1}^n b^i_0\D_i u(x',0) = 0, \quad \sum_{i=1}^n b^i_0\D_{ij} u(x',0) = 0.
\end{equation}
 for each $j = 1,\dots,n-1$.
\end{lem}
\begin{proof}
For $s \neq 0$, the equation $L_0 u = f$ reads
\[
a^{ij}_0 \D_{ij}u + \frac{\alpha}{|s|}\sum_{l=1}^{n-1} b_0^l \D_l u + \frac{\alpha}{s}\D_n u = f
\]
at $(x',s)$. Multiplying by $|s|$ and using $b_0^n = 1$,
\[
\alpha\Big(\sum_{l=1}^{n-1} b_0^l \D_l u(x',s) + \operatorname{sgn}(s)\,\D_n u(x',s)\Big)
= |s|\big(f(x',s) - a^{ij}_0 \D_{ij}u(x',s)\big) \to 0
\]
as $s \to 0$, since $u \in C^{2,\beta}$ and $f \in C^{0,\beta}$ bound the right-hand side. Taking $s \to 0^+$ and $s \to 0^-$ and adding, resp. subtracting, the resulting identities gives
\[
\sum_{l=1}^{n-1} b_0^l \D_l u(x',0) = 0, \qquad \D_n u(x',0) = 0,
\]
and in particular $\sum_{i=1}^n b_0^i \D_i u(x',0) = 0$. Differentiating the first identity tangentially in $x_j$ yields the second identity.
\end{proof}

We proceed to prove a Schauder theorem for \eqref{frozen} in the style of \cite{Simon}.

\begin{thm}\label{constantSchauder}
If $u$ is a viscosity solution to $\mathcal{L}_0 u = f$ in $B_{\rho}(0)$, then for all $\theta \in (0,1)$, we have
\begin{equation}
[u]_{C^{2,\beta}(B_{\theta \rho}(0))} \leq C([f]_{C^{0,\beta}(B_\rho(0))} + \rho^{-2-\beta}\|u\|_{L^\infty(B_\rho(0))}).
\end{equation}
\end{thm}

The above Theorem is a direct consequence of the following a priori estimate.

\begin{lem}
Let $u$ be a $C^{2,\beta}$ solution to $\mathcal{L}_0 u = f$ in $B_{1}(0)$, with $f \in C^{0,\beta}(B_1)$, then 
\begin{equation}
[u]_{C^{2,\beta}(B_{1/2}(0))} \leq C([f]_{C^{0,\beta}(B_1(0))} + \|u\|_{L^\infty(B_1(0))}).
\end{equation}
\end{lem}
\begin{proof}
We prove that 
\begin{equation} \label{deltaschauder}
[D^2u]_{C^{0,\beta}(B_{1/2})} \leq \delta[D^2u]_{C^{0,\beta}(B_{1})} + C([f]_{C^{0,\beta}(B_1)} + \|u\|_{L^\infty(B_1)}).
\end{equation}
Suppose for contradiction that there exist sequences $u_k\in C^{2,\beta}(B_1)$, $f_k \in C^{0,\beta}(B_1)$ such that $\mathcal{L}_0 u_k = f_k$ and for a fixed $\delta_0 > 0$,
\begin{equation} \label{contras}
[D^2u_k]_{C^{0,\beta}(B_{1/2})} > \delta_0[D^2u_k]_{C^{0,\beta}(B_{1})} + k(\|D^2u_k\|_{L^\infty(B_1)}+ [f_k]_{C^{0,\beta}(B_1)}).
\end{equation}
Choose $x_k, y_k \in B_{1/2}$ such that 
\begin{equation}
\frac{|D^2u_k(x_k)-D^2u_k(y_k)|}{|x_k-y_k|^\beta} \geq \frac{1}{2}[D^2u]_{C^{0,\beta}(B_{1/2})},
\end{equation}
and let $\rho_k = |x_k - y_k|$. We first observe that $\rho_k \rightarrow 0$. Since
\begin{align}
\frac{1}{2}[D^2u_k]_{C^{0,\beta}(B_{1/2})} \leq \frac{|D^2u_k(x_k)-D^2u_k(y_k)|}{\rho_k^\beta} \leq \frac{2\|D^2 u_k\|_{L^\infty(B_1)}}{\rho_k^\beta} \leq \frac{2[D^2u]_{C^{0,\beta}(B_{1/2})}}{k\rho_k^\beta},
\end{align}
we have
\begin{equation}
\rho_k \leq Ck^{-1/\beta} \rightarrow 0 \text{ as } k \rightarrow \infty.
\end{equation}
As we rescale and blow up, we distinguish the following two cases.
\begin{align}
(a). &\sup_k x_{k,n}/\rho_k = \infty,\\
(b). &\sup_k x_{k,n}/\rho_k < \infty.
\end{align}
In case (a), we assume without loss of generality $\lim_k x_{k,n}/\rho_k = + \infty$. We may also assume without loss of generality by subtracting a polynomial that
\begin{equation}\label{ukxk'0}
u_k(x_k',0) = |\nabla u_k(x_k',0)| = |D^2 u_k(x_k',0)| = 0.
\end{equation}
Indeed, consider
\begin{equation}
q_k(z) = u_k(x_k',0) + \sum_{i=1}^{n} \D_i u_k(x_k',0) (z_i-x_i) + \frac{1}{2}\sum_{i,j=1}^{n} \D_{ij}u_k(x_k',0)(z_i-x_i)(z_j-x_j).
\end{equation}
Then $L_0 q_k = const.$ by Lemma \ref{Conormal}.

Define 
\begin{align}
&\tilde{u}_k := \frac{u_k(x_k + \rho_kx)-p_k(x)}{\rho_k^{2+\beta}[D^2u_k]_{C^{0,\beta}(B_{1})}},\\
&\tilde{f}_k := \frac{f_k(x_k + \rho_kx)-f_k(x_k)}{\rho_k^{\beta}[D^2u_k]_{C^{0,\beta}(B_{1})}}
\end{align}
where $p_k$ is the quadratic polynomial 
\begin{equation}
p_k(z) = u_k(x_k) + \sum_{i=1}^{n}\rho_k \D_i u_k(x_k) z_i + \frac{1}{2}\sum_{i,j=1}^{n}\rho_k^2 \D_{ij}u_k(x_k)z_iz_j.
\end{equation}
Calculation yields that $\tilde{u}_k$ satisfies in $B_{1/2\rho_k}$,
\begin{align}
a^{ij}_0\D_{ij}\tilde{u}_k(x) + \frac{\alpha b^i_0}{x_{k,n}/\rho_k +x_n}\D_i\tilde{u}_k = \tilde{f}_k(x) + R_k
\end{align}
where $R$ is the error term given by
\begin{align}
R_k = \frac{1}{\rho_k^\beta[D^2u_k]_{C^{0,\beta}(B_{1})}}\left(\frac{x_n \alpha b^i_0}{x_{k,n}(x_{k,n}/\rho_k +x_n)}\D_iu_k(x_k) -\frac{\alpha b^i_0}{x_{k,n}/\rho_k +x_n} \D_{ij}u_k(x_k)x_j \right).
\end{align}
Here we used that $x_{k,n}/\rho_k +x_n > 0 $ and dropped the absolute values. Denoting $\xi_k = \frac{x_k-y_k}{\rho_k}$, we note that 
\begin{equation}
[D^2\tilde{u}_k]_{C^{0,\beta}(B_{1/2\rho_k})} \leq 1,\quad  \left|D^2 \tilde{u}_k\left(\xi_k\right)\right| > \delta_0/2.
\end{equation}
Since $\tilde{u}_k$ are uniformly bounded in compact subsets, and bounded in $C^{2,\beta}$ norm, by Ascoli-Arzela that the sequence $\tilde{u}_k$ converges, up to a subsequence to a $C^{2,\beta}$ function $\tilde{u}$ on compact subsets of $\RR^n$. Moreover, $\xi_k\rightarrow \xi$ up to a subsequence. It follows that
\begin{align}
&\tilde{u}(0) = |D\tilde{u}(0)| = |D^2 \tilde{u}(0)| = 0, \\&[D^2\tilde{u}]_{C^{0,\beta}(\RR^n)} \leq 1,\quad \left|D^2 \tilde{u}\left(\xi\right)\right| > \delta_0/2.
\end{align}
Now for any $R \geq 1$, we have
\begin{align}
\|\tilde{f}_k\|_{L^\infty(B_R)} &= \sup_{x \in B_R}\frac{|f_k(x_k + \rho_k x)- f_k(x_k)|}{\rho_k^\beta[D^2u_k]_{C^{0,\beta}(B_1)}} \leq \frac{\rho_k^\beta R^\beta[f_k]_{C^{0,\beta}(B_1)}}{\rho_k^\beta[D^2u_k]_{C^{0,\beta}(B_1)}}
\\&\leq \frac{R^\beta [D^2u_k]_{C^{0,\beta}(B_{1/2})}}{k[D^2u_k]_{C^{0,\beta}(B_1)}} \leq \frac{R^\beta}{k}\rightarrow 0.
\end{align}
To estimate $R_k$, we first note that by \eqref{ukxk'0}
\begin{align}
|D^2 u_k(x_k)| \leq [D^2u_k]_{C^{0,\beta}(B_1)} x_{k,n}^\beta, \quad |D u_k(x_k)| \leq [D^2u_k]_{C^{0,\beta}(B_1)} x_{k,n}^{1+\beta}.
\end{align}
Moreover, we note that for $x_{k,n}/\rho_k > R$, 
\begin{align}
\frac{x_{k,n}}{\rho_k} + x_n \geq \frac{x_{k,n}}{\rho_k}  - R \geq \frac{1}{2}\frac{x_{k,n}}{\rho_k}.
\end{align}
Therefore, we have
\begin{align}
|R_k| &\leq CR\frac{1}{\rho_k^\beta [D^2u_k]_{C^{0,\beta}(B_1)}} \left(\frac{2\rho_k}{x_{k,n}}[D^2u_k]_{C^{0,\beta}(B_1)} x_{k,n}^{\beta} + \frac{2\rho_k}{x_{k,n}}[D^2u_k]_{C^{0,\beta}(B_1)} x_{k,n}^{\beta} \right) \\&\leq CR\left(\frac{x_{k,n}}{\rho_k}\right)^{\beta-1} \rightarrow 0.
\end{align}
We also have
\begin{align}
\left|\frac{\alpha b^i_0}{x_{k,n}/\rho_k +x_n}\D_i\tilde{u}_k \right| \leq C \frac{\rho_k}{2x_{k,n}}R^{1+\beta} \rightarrow 0
\end{align}
Thus taking $k \rightarrow \infty$, we have
\begin{equation}
a^{ij}_0\D_{ij}\tilde{u} = 0 \text{ in } \RR^n.
\end{equation}
Since $[D^2\tilde{u}]_{C^{0,\beta}(\RR^n)} \leq 1$, $|\tilde{u}(x)| \leq C|x|^{2+\beta}$, and thus by the usual Liouville theorem for uniformly elliptic equations, $\tilde{u}$ is a polynomial of degree $2$ and $D^2\tilde{u}$ is a constant. Since $D^2\tilde{u}(0) = 0$, $D^2\tilde{u} = 0$, contradicting $|D^2\tilde{u}(\xi)| > \delta_0/2$. This proves \eqref{deltaschauder} in Case 1.

Now suppose Case 2 holds.
Define 
\begin{align}
&\tilde{u}_k := \frac{u_k((x_k',0) + \rho_kx)-p_k(x)}{\rho_k^{2+\beta}[D^2u_k]_{C^{0,\beta}(B_{1})}},\\
&\tilde{f}_k := \frac{f_k((x_k',0) + \rho_kx)-f_k(x_k',0)}{\rho_k^{\beta}[D^2u_k]_{C^{0,\beta}(B_{1})}}
\end{align}
where $p_k$ is the quadratic polynomial 
\begin{equation}
p_k(z) = u_k(x_k',0) + \sum_{i=1}^{n}\rho_k \D_i u_k(x_k',0) z_i + \frac{1}{2}\sum_{i,j=1}^{n}\rho_k^2 \D_{ij}u_k(x_k',0)z_iz_j.
\end{equation}
We have in $B_{1/2\rho_k}$, 
\begin{equation}
a^{ij}_0\D_{ij} \tilde{u}_k + \frac{\alpha b^l_0}{|x_n|}\D_{l} \tilde{u}_k + \frac{\alpha}{x_n}\D_{n} \tilde{u}_k = \tilde{f}_k.
\end{equation}
The rest of the argument is the same as in Case 1 where we replace the Liouville theorem by Theorem \ref{Liouville}. Hence, we have a contradiction.

Now \eqref{deltaschauder} follows from the interpolation inequality
\begin{equation}
\|D^2u\|_{L^\infty(B_1)} \leq \eta[D^2u]_{C^{0,\beta}(B_1)} + C_\eta\|u\|_{L^\infty(B_1)} \text{ for any } \eta > 0.
\end{equation}
Finally, the conclusion of the theorem follows from the following iteration lemma.
\end{proof}

\begin{lem} \label{iterationlemma} \cite{Simon}
Let $S$ be a monotone subadditive function on the convex subsets of a ball $B = B_{\rho_0}(y_0)$ and let $\theta_0 \in (0,1/2]$, $k > 0$ be given constants. Then there exists $\eps = \eps(\theta_0,k)\in (0,1)$ such that if $E\geq 0$ is a constant and 
\begin{equation}
\sigma^k S(B_{\theta_0\sigma}(y_0)) \leq \eps \sigma^k S(B_{\sigma}(y_0)) + E
\end{equation}
for all balls $B_\sigma(y) \subset B$, then for any ball $B_\rho \subset B$, we have
\begin{equation}
\rho^k S(B_{\theta \rho}(y)) \leq CE
\end{equation}
for all $\theta \in (0,1)$ , where $C = C(n,\theta_0,\theta,k)$.
\end{lem}

\section{Improvement of Flatness} \label{Sec:Improvement}

This section establishes the following Theorem.
\begin{thm}
Let $v$ be a viscosity solution to \eqref{vpde}, then 
$v(\cdot,0) \in C^{1,\sigma}(B_1')$, in particular, $v$ is singlevalued.
\end{thm}

The above theorem follows from a standard iteration of the following improvement of flatness lemma.

\begin{lem}\label{improveflat}
Let $v$ be a viscosity solution to \eqref{vpde} and satisfies 
\begin{equation}
x_n - \eps \leq v \leq x_n + \eps \text{ in } B_1^+(0),
\end{equation}
then there exists $0<r \leq r_0$ for a universal $r_0$ and $0 < \eps \leq \eps_0$ for some $\eps_0$ depending on $r$, and $q' \in \RR^{n-1}$ such that,
\begin{equation}
x_n - \frac{r\eps}{2} \leq v - \eps q'\cdot x' \leq x_n  + \frac{r\eps}{2} \text{ in }B_{r}^+.
\end{equation}

\end{lem}
\begin{proof}
Fix $r \leq r_0$. Suppose for contradiction that there exist a sequence $\eps_k \rightarrow 0$, and a sequence $v_k$ of solutions to \eqref{vpde} in $B_1$ such that the conclusion of the Theorem is not true.

Consider
\begin{equation}
{w}_k= \frac{{v}_k(x) - x_n}{\eps_k}  \text{ in } B_1.
\end{equation}
By assumption,
\begin{equation}
|{w}_k| \leq 1.
\end{equation}

Define 
\begin{equation}
\mathcal{F}_k(M,P):= \eps_k^{-1} [{F}({G}((D^2 M, DP)) - {F}({G}(D^2x_n,Dx_n))]
\end{equation}

It follows from \eqref{vpde} that,
\begin{equation}
\mathcal{F}_k(D^2{w}_k,D{w}_k) = D{{F}}({G}(D^2x_n,Dx_n)):[D^2 {w}_k,D{w}_k ] + o(\eps_k D^2 {w}_k, \eps_k D{w}_k).
\end{equation}

Moreover, Corollary \ref{harnackcor} gives
\begin{equation}
|{w}_k(x)-{w}_k(y)|\leq C|x-y|^\sigma \text{ in }B_{1/2}
\end{equation}
for some universal $C$ and $|x-y| \geq \frac{\eps_k}{\eps_0}$.

By Ascoli-Arzela, as $\eps_k \rightarrow 0$, we can extract a convergent subsequence that converges to a H\"older continuous function in $B_{1/2}$, and 
\begin{align}
&{F}_{k} \rightarrow {F}_{*},\ D{F}_{k} \rightarrow D{F}_{*} \text{ uniformly, }\\
&{w}_k \rightarrow {w}_*, \text{ uniformly in compact subsets of $B_{1/2}^+$, }
\end{align}

We first show that 
\begin{equation}
L_*{w}_*:= D{{F}}({G}(D^2x_n,Dx_n)):[D^2 {w}_*,D{w}_* ] = 0 \text{ in } B_{1/2}^+.
\end{equation}

Assume by contradiction that we touch $w_*$ from above at $x_*$ by a smooth function $\phi$, and 
\begin{equation}
D{{F}}({G}(D^2x_n,Dx_n)):[D^2 \phi,D\phi](x_*) > c > 0.
\end{equation}
Then $\phi + const.$ touches $w_k$ from above at $x_k$ and $x_k \rightarrow x_*$. We have
\begin{align}
0 &\geq D{{F}}({G}(D^2x_n,Dx_n)):[D^2 \phi,D\phi ](x_k) + o(\eps_k D^2 \phi(x_k), \eps_k D\phi(x_k))\\
& \geq D{{F}}({G}(D^2x_n,Dx_n)):[D^2 \phi,D\phi ](x_k) > c/2 \text{ as } k \rightarrow \infty,\end{align}
which is a contradiction.

Extend ${w}_*$ evenly across $x_n = 0$ and denote the extended function still by ${w}_*$. It follows from Theorem \ref{evenextension} that ${w}_*$ is a viscosity solution to 
\begin{equation} \label{limitpde}
L_*{w}_*:= D{{F}}({G}(D^2x_n,Dx_n)):[D^2 {w}_*,D{w}_* ] = 0 \text{ in } B_{1/2}.
\end{equation}

By Theorem \ref{constantSchauder}, ${w}_*\in C^2(B_{1/2})$ and there exists
\begin{equation}
P(x',x_n) = a + bx_n + q' \cdot x' \text{ in }B_{1/2}.
\end{equation}
such that
\begin{align}
|{w_*}(x',x_n) - P| \leq C|x|^{2} \text{ in }B_{1/2}.
\end{align}
We note that $a = w_*(0) = 0$, and $b = \tilde{w_*}_n(0) = 0$ since $w$ is even. For $k$ large, we have
\begin{equation}
\left| \frac{v_k-x_n}{\eps_k}- q' \cdot x'\right| \leq Cr^{2} \text{ in } B_{r}.
\end{equation}
which can be rewritten as
\begin{equation}
x_n - C\eps_k r^{2} \leq v_k- {\eps_k} q' \cdot x' \leq x_n + C\eps_k r^{2} \text{ in } B_{r}.
\end{equation}
which is a contradiction for choosing $r_0$ small, and the Lemma follows.
\end{proof}

\section{Higher Boundary Regularity} \label{Sec:Higher}

This section first establishes that $v\in C^{2,\beta}$. It follows that the coefficients as shown in $\eqref{coef1},\eqref{coef2}$ are continuous where we can then apply a variable coefficient Schauder theory iteratively.

\begin{thm}\label{vc2alpha}
Let $v$ be a solution to \eqref{vpde} in $B_1(0)^+$, then $v \in C^{2,\beta}(B^+_{1/2}(0))$. 
\end{thm}
\begin{proof}
Suppose for contradiction that there exists a sequence $\eps_k \rightarrow 0$ such that the conclusion of the theorem doesn't hold.

Let $v_0 = x_n$, $\phi = v-v_0$, then $\phi$ satisfies
\begin{equation}
\tilde{F}(D^2\phi,D\phi) :=F(G(D^2v_0+\D^2\phi,Dv + D\phi))  - F(G(D^2v_0,Dv_0)) = 0
\end{equation}
Define
\begin{equation}
P(a,b,c_i,d_i,f_{ij},x) = a x_n + b x_n^2 + \sum_{i = 1}^{n-1}c_ix_ix_n + \sum_{i = 1}^{n-1} d_i x_i^2 + \sum_{i,j = 1}^{n-1} f_{ij} x_ix_j.
\end{equation}
We sometimes write $P(x) = P(a,b,c,d,x)$ as $a,b,c,d$ are constants which could possibly change in different lines. The above theorem easily follows from the following Claim. 

\textit{Claim.}
There exist small constants $\eta$ such that if 
\begin{equation} \label{phiellBr}
\|\phi - P(x)\|_{L^\infty\left(\overline{B_r^+}\right)} \leq r^{2+\beta},
\end{equation}
then 
\begin{equation}
\|\phi -  P(x)\|_{L^\infty(\overline{B_{\eta r}^+})}  \leq {(\eta r)}^{2+\beta}.
\end{equation}

For $x \in \overline{B_r^+}$, let $w: \overline{B_r^+} \rightarrow [-1,1]$ be such that
\begin{equation}
\phi(x) = P(x) + r^{2+\beta}w(x/r).
\end{equation}
We wish to show that 
\begin{equation}
\|w - P(x)\|_{L^\infty(\overline{B_1^+})} \leq \eta^{2+\beta}.
\end{equation}
We observe that $w$ satisfies 
\begin{equation} \label{prweqn}
\tilde{F}(D^2P + r^{\beta} D^2 w,DP + r^{1+\beta}Dw) = 0.
\end{equation}
Define 
\begin{equation}
\mathcal{F}_r(M,p):= r^{-\beta} [\tilde{F}(D^2P + r^{\beta} D^2 w,DP + r^{1+\beta}Dw) - \tilde{F}(D^2P,DP)]
\end{equation}

We note that by Corollary \ref{harnackcor} applied to \eqref{prweqn}, $w$ is H\"older in compact subsets of $B_1^+$. Now we proceed with a compactness argument. Suppose by contradiction that the Lemma is false. Then there exists a sequence of $r_k \rightarrow 0$ and corresponding $\mathcal{F}_{r_k},\tilde{F}_k,w_k,a_k,b_k,c_{i,k},d_{i,k},f_{ij,k}$ for which the conclusion doesn't hold. We can extract a convergent subsequence that converges to a H\"older continuous function in $B_{1/2}$, and 
\begin{align}
&\tilde{F}_{k} \rightarrow \tilde{F}_{*},\ D\tilde{F}_{k} \rightarrow D\tilde{F}_{*} \text{ uniformly, }\\
&w_k \rightarrow w_*, \text{ uniformly in compact subsets of $B_{1/2}^+$, }\\
&a_k \rightarrow a_*,\ b_k \rightarrow b_*,\ c_{i,k} \rightarrow c_{i,*}, \ d_{i,k} \rightarrow d_{i,*}, \ f_{ij,k} \rightarrow f_{ij,*}
\end{align}
and $P_* = a_*x_n + b_* x_n^2 + \sum_{i = 1}^{n-1}c_{i,*}x_ix_n + \sum_{i = 1}^{n-1}d_{i,*}x_i^2 + \sum_{i,j = 1}^{n-1}f_{ij,*}x_ix_j$.

Extend $w_*$ evenly across $x = 0$ and denote the extended function still by $w_*$. We claim that $w_*$ is a solution to 
\begin{equation} \label{limitpde}
L_*w_*:= D\tilde{F}(D^2P_*,DP_*):[D^2 w_*, D w_* ]= 0 \text{ in } B_1.
\end{equation}
We first show that $w_*$ is a viscosity solution of \eqref{limitpde} away from $\{x_n = 0\}$. 

Assume by contradiction that we touch $w_*$ from above at $x_*$ by a smooth function $\phi + \eps|x-x_*|^2$, and 
\begin{equation}
D\tilde{F}(D^2P_*,DP_*):[D^2 \phi(x_*),D \phi(x_*)] > \eps > 0.
\end{equation}
Then $\phi + const.$ touches $w_k$ from above at $x_k$ and $x_k \rightarrow x_*$. We have
\begin{align}
0 &\geq r^{-\beta} [\tilde{F}(D^2P_* + r^{\beta} D^2 \phi,DP_* + r^{1+\beta}D\phi) - \tilde{F}(D^2P,DP)]\\
& \geq D\tilde{F}(D^2P_*,DP_*):[D^2 \phi(x_k),D \phi(x_k)] - C \max\{r^{\beta}\} > \eps/2 \text{ as } k \rightarrow \infty,
\end{align}
which is a contradiction. Here we used the following linearization
\begin{align} \nonumber
&\tilde{F}_k(D^2P_k + r_k^{\beta} D^2 \phi,DP_k+ r_k^{1+\beta}D\phi) = \tilde{F}_k(D^2P_k,DP_k) + D\tilde{F}_k(D^2P_k,DP_k):[r_k^{\beta} D^2 \phi,r_k^{1+\beta}D\phi] \\& \hspace{2.5 in} + o(r_k^{\beta} D^2 \phi, r_k^{\beta}D\phi).
\end{align}
We point out here that the first order coefficient in $D\tilde{F}_k(D^2P_k,DP_k)$ carries a factor of $1/r_k x_n$ which cancels with the extra factor of $r$ in the gradient term $r_k^{1+\beta}D\phi$.

Since $v, v_0, P$ are smooth away from $\{x_n = 0\}$, $w$ is a classical solution of \eqref{limitpde} away from $\{x_n = 0\}$ and thus $w$ is a viscosity solution by Theorem \ref{evenextension}. Therefore, Theorem \ref{constantSchauder} applies, and we have a contradiction.
\end{proof}

\begin{thm}\label{Schauder}
Let $w$ be a viscosity solution to $|x_n|^\alpha Lu = f$ in $B_\rho^+(0)$, then for each $\theta \in (0,1)$, we have
\begin{equation}
[w]_{C^{2,\beta}(B_{\theta\rho}^+(0))} \leq C([f]_{C^{0,\beta}(B_\rho^+(0))}) + \rho^{-2-\beta}\|w\|_{L^\infty(B_\rho^+(0))}).
\end{equation} 
\end{thm}
\begin{proof}
By Theorem \ref{evenextension} we extend $w$ by even reflection, and extend the operator $L$ correspondingly, then $w$ is a viscosity solution in $B_\rho(0)$. We note that for each $y \in B_{\rho}(0)$ and each $\sigma \in (0,\rho)$ such that $B_\sigma \subset B_{\rho}(0)$, we can write 
\begin{equation}
|x_n|^\alpha L = \mathcal{L}_0 - R
\end{equation}
where $\mathcal{L}_0$ is the frozen coefficient operator \eqref{frozen}, and 
\begin{equation} \label{diff}
\mathcal{L}_0w = Rw + f
\end{equation}
and 
\begin{align}
|{R}w| \leq |x_n^\alpha(a^{ij}(x) - a^{ij}_0)w_{ij}| &+ |x_n^\alpha(a^{nn}(x) - a^{nn}_0)w_{nn}| + 2|x_n^\alpha(a^{in}(x) - a^{in}_0)w_{in}| \\&+ \left|\alpha x_n^{\alpha-1}(b^i(x) - b^{i}_0)w_i\right| + \left|\alpha x_n^{\alpha-1}(b^n(x) - b^{n}_0)w_n\right|.
\end{align}
where $a^{ij}$, $a^{in}$, $a^{nn}$, $b^i$, $b^n$ is given explicitly in \eqref{coef1} and \eqref{coef2} and are continuous by Theorem \ref{vc2alpha}. Therefore,
\begin{equation}
[{R}w]_{C^{0,\beta}(B_{\theta\sigma}(0))} \leq \frac{\delta}{2}[w]_{C^{2,\beta}(B_{\theta\sigma}(0))}
\end{equation}

It follows from Theorem \ref{constantSchauder} that
\begin{equation}
[w]_{C^{2,\beta}(B_{\sigma/2}(0))} \leq \delta [w]_{C^{2,\beta}(B_{\sigma})} + C([f]_{C^{0,\beta}(B_\sigma)} + \sigma^{-2-\beta}\|w\|_{L^\infty(B_\sigma)}).
\end{equation}
Now Lemma \ref{iterationlemma} finishes the proof of the theorem.
\end{proof}

\begin{proof}[Proof of Theorem \ref{ThmA}]
Let $v$ be the hodograph transform of $u$ defined by \eqref{hodo}. It follows from the explicit expression of \eqref{wpde1} and Theorem \ref{vc2alpha} that the coefficient of \eqref{wpde1} are $C^{0,\sigma}$. Applying Theorem \ref{Schauder} to \eqref{wpde1}, we have $ w = v_e \in C^{2,\beta}(\overline{B^+_{3/4}(0)})$, and applying this to every tangential directions, we have
\begin{equation}\label{vc3}
v \in C_{x'}^{3,\beta}C^{2,\beta}_{x_n}(\overline{B^+_{3/4}(0)}).
\end{equation}
Now we can differentiate the equation \eqref{wpde1} in tangential direction again and get an equation for $w_{e_1}$ and $w_{e_1}$ satisfies a PDE of the same form as \eqref{wpde1}, for $1\leq i,j \leq n$, 
\begin{equation}\label{2dervpde}
x_n^\alpha a^{ij}(x)w_{e_1ij} + x_n^{\alpha-1} \alpha b^i(x) w_{e_1i} = -x_n^\alpha a^{ij}_{e_1}(x)w_{ij}- \alpha x_n^{\alpha -1}b^i_{e_1}(x) w_{i}:= f.
\end{equation}

We first observe that by \eqref{vc3} and the explicit formula of the coefficient \eqref{coef1}, \eqref{coef2}, $a^{ij}_{e_1}, b^i_{e_1}$ are H\"older continuous. Therefore, we can apply Theorem \ref{Schauder} again and get $w_{e_1} \in C_{x'}^{2,\beta}(\overline{B^+_{5/8}})$ which implies 
\begin{equation}
v \in C_{x'}^{4,\beta}C^{2,\beta}_{x_n}(\overline{B^+_{5/8}(0)}).
\end{equation}
Now it is clear that we can iterate the process and get after $k$ iterations,
\begin{equation}
v \in C_{x'}^{2+k,\beta}C^{2,\beta}_{x_n}(\overline{B^+_{2^{-1} (1+ 2^{-k})}(0)}),
\end{equation}
which yields
\begin{equation}
v \in C_{x'}^{\infty}C^{2,\beta}_{x_n}(\overline{B^+_{1/2}}).
\end{equation}
Finally the Theorem follows by noticing that the free boundary $\D \{u > 0\}$ is the graph of $v(x',0)$.

\end{proof}

\section*{Acknowledgement}
The author gratefully acknowledges his PhD advisor, Dennis Kriventsov, for many enlightening conversations regarding this work. This work is partially supported by NSF Division Of Mathematical Sciences grant DMS-2247096.\\

\nocite{Caffarelli1989, HanLin, MR1658612, MR1986693, MR643158, MR4093736, MR4821888, MR3310271, MR4308249, MR1118699, MR973745, MR4309882}

\bibliographystyle{plain}
\bibliography{bib.bib}
\end{document}